\documentclass[12pt]{amsart}
\usepackage{amssymb,latexsym}
\usepackage{amsfonts,amsmath,amsthm}
\usepackage[dvips]{graphicx}
\usepackage[latin1]{inputenc}
\newtheorem*{teoa}{Theorem A}

\newtheorem{defi}{Definition} 
\newtheorem{teo}{Theorem} 
\newtheorem{pro}{Proposition} 

\newtheorem{lema}{Lemma}
\newtheorem{rem}{Remark}
\newcommand{\N}{\mathbb{N}}

\newcommand{\R}{\mathbb{R}}

\begin{document}

\title[A generalization of Aubin's result]{A generalization of Aubin's result for a Yamabe-type problem on smooth metric measure spaces}
\author{Jhovanny Mu\~{n}oz Posso \textsuperscript{1} \textsuperscript{2}} 

\address{$^{1}$ Departamento de Matem\'aticas,
\small{Universidad del Valle, Calle $13 \# 100-00,$ Cali, Colombia}}
\email{jhovanny.munoz@correounivalle.edu.co}

\address{$^{2}$ Instituto de Matem\'aticas Puras e Aplicadas, Estrada Dona Castorina 110. Rio de Janeiro, Brasil.}
\email{jhovamu@impa.br}



\maketitle

\begin{abstract}
The Yamabe problem in compact closed Riemannian manifolds is concerned with finding a metric with constant scalar curvature in the conformal class of a given metric. This problem was solved by the combined work of Yamabe, Trudinger, Aubin, and Schoen. In particular, Aubin solved the case when the Riemannian manifold is compact, is nonlocally conformally flat and has a dimension equal to or greater than $6$. In $2015$, Case considered a Yamabe-type problem in the setting of smooth measure space in manifolds and for a parameter $m$, which generalizes the original Yamabe problem when $m=0$. Additionally, Case solved this problem when the parameter $m$ is a natural number. In the context of the Yamabe-type problem, we generalize Aubin's result for nonlocally conformally flat manifolds, with dimension equal and greater than 6 and  parameter $m$ close to nonnegative integers.

\end{abstract}

\section{Introduction}\label{intro}
Let $(M^n, g)$ be an $n$-dimensional compact Riemannian manifold and $R_g$ be the scalar curvature associated with the metric $g$. The Yamabe problem is concerned with finding a metric of constant scalar curvature in the conformal class of $g$. It is well known that the Yamabe problem was solved by the combined work of Yamabe \cite{Yamabe}, Trudinger \cite{Trudinger}, Aubin \cite{Aubin}, and Schoen \cite{Schoen}; for a presentation of this topic, see \cite{LeeParker}. In particular, we mention that Aubin in \cite{Aubin} solved the problem under the hypothesis that the Riemannian manifold is compact, nonlocally conformally flat and with dimension $n \geq 6$.

In \cite{CaseYamabe} and \cite{CaseGNS}, Case considered some geometric invariants that he called the \textit{weighted Yamabe constants}, which constitute a one-parameter family and interpolate between the Yamabe constant and Perelman's $\nu$-entropy when the parameter is zero and infinity, respectively. The Yamabe constant is the curved analogue of the sharp Sobolev inequality, and the $\nu$-entropy is the curved analogue of the sharp logarithmic Sobolev inequality. The weighted Yamabe constants are thus curved analogues of a family of sharp Gagliardo-Nirenberg-Sobolev inequalities. 






Before we explain Case's results, we introduce some terminology. Let us denote by $dV_g$ the volume form induced by the metric $g$. Fix a function $\phi \in C^{\infty}(M)$ and a dimensional parameter $m \in [0, \infty]$. A \textit{smooth metric measure space} is a four-tuple $(M^{n}, g, e^{-\phi} dV_g,m)$. Let us denote by $\Delta_g$ and $\nabla_g$ the Laplacian and the gradient associated to the metric $g$, respectively. The \textit{weighted scalar curvature} $R^{m}_{\phi}$ of a smooth metric measure space for $m=0$ is $R^{m}_{\phi} = R_g$ and for $m \neq 0$ is the function $R^{m}_{\phi} := R_g + 2\Delta_g \phi- \frac{m+1}{m}|\nabla_g \phi|^{2}$. The \textit{weighted Yamabe quotient} is the functional $\mathcal{Q}: C^{\infty} (M) \to \R$ defined by
\small
\begin{equation}\label{funcional Yamabe}
\begin{array}{ll}
\mathcal{Q}(w) = & \dfrac{ \int_{M}(|\nabla_g w|^{2} + \frac{m + n - 2}{4(m + n - 1)}R^{m}_{\phi } w^{2}) e^{-\phi}dV_g \left( \int_{M}| w| ^{\frac{2(m + n -1)}{m +n -2}} e^{-\frac{(m-1)\phi}{m}}dV_g \right)^{\frac{2m }{n}} }{ \left(\int_{ M } |w|^{\frac{2(m + n )}{m +n -2}} e^{-\phi} dV_g \right)^{
\frac{2m + n -2}{ n}} }.
\end{array}  
\end{equation}

\normalsize

The \textit{weighted Yamabe constant} is the number
\begin{equation}\label{Defi Yamabe constant}
\Lambda[M^n, g, e^{-\phi} dV_g, m] = \inf \{  \mathcal{Q}(w) : w \in  C^{\infty}(M), w \not \equiv 0 \}.
\end{equation}

For $m = \infty$, Case defined the weighted Yamabe quotient as the limit of \eqref{funcional Yamabe} when $m$ goes to infinity and the weighted Yamabe constant as \eqref{Defi Yamabe constant}. Note that when $m = 0$, the weighted Yamabe constant coincides with the Yamabe constant, and when $m = \infty$, if the weighted Yamabe constant is positive, this is equivalent to Perelman's entropy (see \cite{Perelman}). By the Gagliardo-Nirenberg inequality, it is possible to consider the functional defined in \eqref{funcional Yamabe} and the constant defined in \eqref{Defi Yamabe constant} for $(\R^n, dx^2, dV, m)$, where $dx^2$ is the usual metric in $\R^2$ and $dV$ is the volume form induced by this metric. We denote by $\Lambda_{m,n}$ the weighted Yamabe constant $\Lambda(\R^n, dx^2, dV, m)$ (see Theorem \ref{teoGN} below).

Two smooth measure metric spaces $(M^n, g, e^{- \phi} dV_g, m )$ and
$(M^n, \hat{g}, e^{- \hat{\phi}} dV_{\hat{g}}, m )$ are pointwise conformally equivalent if there is a function $\sigma \in C^{\infty}(M)$ such that $\hat{g} = e^{\frac{2 \sigma}{m + n - 2}}g$ and $\hat{\phi} = \frac{-m\sigma}{m+n-2} + \phi$. We say that $(M^n, g, e^{- \phi} dV_g,  m )$ and $(\hat{M}^n, \hat{g}, e^{- \hat{\phi}} dV_{\hat{g}},  m )$ are conformally equivalent if there is a diffeomorphism $F : \hat{M} \to M$ such that $(\hat{M}^n, \hat{g}, e^{- \hat{\phi}} dV_{\hat{g}},  m )$
is pointwise conformally equivalent to $(F^{-1}(M), F^{*}g, F^{*}(e^{- \phi} dV_g),  m )$.

The \textit{weighted Yamabe problem} is to find a function that minimizes the weighted Yamabe quotient. In \cite{CaseYamabe}, Case proved an Aubin-type criterion for the existence of a minimizer of the weighted Yamabe quotient. The exact statement is:

\begin{teo}[\cite{CaseYamabe}]\label{Aubin Case}
Let $(M^n, g, e^{-\phi} dV_g, m)$ be a compact smooth metric measure space such that $m \geq 0$. Then,
\begin{equation}\label{desigualdad Aubin}
\Lambda[M^n, g, e^{-\phi} dV_g, m] \leq \Lambda_{m,n}.
\end{equation}

Furthermore, if the inequality \eqref{desigualdad Aubin} is strict, then there exists a smooth positive function $w$ such that 
\begin{equation*}\label{equality}
\mathcal{Q}(w) = \Lambda[M^n, g, e^{-\phi} dV_g, m].
\end{equation*}

\end{teo}

Additionally, in \cite{CaseYamabe}, Case proved the strict inequality in \eqref{desigualdad Aubin} when $m \in  \N \cup \{0\} $, together with a characterization for the equality in \eqref{desigualdad Aubin}.

\begin{teo}[\cite{CaseYamabe}]\label{Case teo}
Let $(M^n, g, e^{-\phi} dV_g, m)$ be a compact smooth metric measure space such that $m \in \N \cup \{0\}$. The following equality holds
\begin{equation}\label{igualdad Aubin m entero}
\Lambda[M^n, g, e^{-\phi} dV_g, m] = \Lambda_{m,n}
\end{equation} if and only if $(M^n, g, e^{-\phi} dV_g, m)$ is conformally equivalent to $(S^n, g_0, dV_{g_0}, m)$, where $(S^n, g_0)$ is the n-dimensional sphere with a metric of constant sectional curvature  and $m \in \{0, 1\}$. Therefore, there exists a smooth positive function $w$ such that 
\[
\mathcal{Q}(w) = \Lambda[M^n, g, e^{-\phi} dV_g, m].
\]
\end{teo}



In this paper, we prove for nonlocally conformally flat manifolds with dimension $n \geq 6$ and $m \in  \bigcup \limits_{i \in \mathbb{N} \cup \{ 0 \}} [i,i + \delta)$ for some $\delta > 0$ that inequality \eqref{desigualdad Aubin} is strict. By Theorem \ref{Aubin Case}, the existence of a minimizer of the weighted Yamabe problem follows. This result is a ge\-ne\-ra\-li\-za\-tion of the Aubin existence theorem and a generalization of the Case existence result for $m$ close to the integers.

\begin{teoa}\label{minimizante Jhova Aubin}
Let $(M^n,  g)$ be a compact Riemannian manifold, $\phi \in C^{\infty}(M)$, and $n \geq 6$. If $(M,g)$ is nonlocally conformally flat, there exist $ 0 < \delta \leq 1$ such that for 
\[m \in  \bigcup \limits_{i \in \mathbb{N} \cup \{ 0 \}} [i,i + \delta)\] we obtain
\begin{equation}\label{desigualdad Aubin locally flat}
\Lambda[M^n, g, e^{-\phi} dV_g, m] < \Lambda_{m,n}. 
\end{equation}
 
Therefore, there exists a smooth positive minimizer of the weighted  Yamabe quotient.
\end{teoa}

We use arguments similar to those Aubin used in \cite{Aubin} to prove Theorem A. These arguments involve test functions in the Yamabe quotient with support in a neighborhood of a point where the Weyl tensor is nonzero. However, when we restrict to the case $m=0$, we use different test functions from the ones used in \cite{Aubin}. For this reason, our proof is different than Aubin's Theorem for $n\geq 7$.

This paper is organized as follows. In sections \ref{smooth} and \ref{Yamabe}, we present some basic concepts about smooth metric measure spaces and the Yamabe-type problem on these spaces, respectively. In section \ref{Aubin}, we prove Theorem A. We complement this with an appendix devoted to some calculus lemmas that we use in the proof of Theorem A.

\section{Smooth metric measure space and the conformal Laplacian}\label{smooth}

Our approach in this section is based on \cite{CaseYamabe} and \cite{CaseGNS}. The first step is to introduce the definition of a smooth metric measure space 

\begin{defi} Let $(M^n, g)$ be a Riemannian manifold, and let us denote by $dV_g$ the volume form induced by $g$ in $M$. Fix a function $\phi \in C^{\infty}(M)$ and a dimensional parameter $m \in [0, \infty]$. When $m = 0$, we require that $\phi = 0$. A smooth metric measure space is the four-tuple \hbox{$(M^n, g, e^{- \phi} dV_g,  m )$}.
\end{defi}

As in \cite{CaseYamabe}, we sometimes denote by the three-tuple $(M^n, g, v^{m} dV_g)$ a smooth metric measure space where $v$ and $\phi$ are related by $v^{m} = e^{- \phi}$. We denote by $R_g$, $Ric_g$, $T_g$ and $W_g$ the scalar curvature, the Ricci tensor, traceless Ricci tensor and the Weyl tensor of $(M, g)$, respectively. In some cases, we omit the reference to the metric $g$ and write $R$, $Ric$, $T$, and $W$. In the following definitions, we consider the case of $m = \infty$ as the limit case of the parameter $m$.

\begin{defi}
Given a smooth metric measure space $(M^n, g, e^{- \phi} dV_g,  m )$, the   weighted scalar curvature $R^{m}_{\phi}$ and the Bakry-\'Emery Ricci curvature $Ric^{m}_{\phi}$ are the tensors
\begin{equation}\label{Bakrycurvature}
R^{m}_{\phi} := R_g + 2 \Delta_g \phi - \frac{m + 1}{m} |\nabla_g \phi|_g^2
\end{equation} and
\begin{equation}
Ric^{m}_{\phi} := Ric +  Hess_g \phi - \frac{1}{m} d \phi \otimes d\phi,
\end{equation} where $\Delta_g$ is the usual Laplacian, $\nabla_g $ is the gradient, $|\cdot|_g$ is the tensor norm, and $Hess_g$ is the Hessian, all of which are calculated in the metric $g$.
\end{defi}

Since $\phi = -m \ln v$, equality \eqref{Bakrycurvature} takes the form of 

\begin{equation}\label{Bakryv}
  R^{m}_{\phi} = R_g + \frac{m(1-m)}{v^2} |\nabla_g v|_g^2 - \frac{2m}{v} \Delta_g v. 
\end{equation}

\begin{defi}\label{conformal definition}
Let $(M^n, g, e^{- \phi} dV_g, m )$ and
$(M^n, \hat{g}, e^{- \hat{\phi}} dV_{\hat{g}}, m )$ be smooth metric measure spaces. We say that they are pointwise conformally equivalent if there is a function $\sigma \in C^{\infty}(M)$ such that
\begin{equation}\label{conformal relation}
\hat{g} = e^{\frac{2 \sigma}{m + n - 2}}g \quad \text{and} \quad  \hat{\phi} = \frac{-m\sigma}{m+n-2} + \phi.
\end{equation} 
$(M^n, g, e^{- \phi} dV_g,  m )$ and
$(\hat{M}^n, \hat{g}, e^{- \hat{\phi}} dV_{\hat{g}},  m )$ are conformally equivalent if there
is a diffeomorphism $F : \hat{M} \to M$ such that 
$(F^{-1}(M), F^{*}g, F^{*}(e^{- \phi} dV_g),  m )$ is pointwise conformally equivalent to $(\hat{M}^n, \hat{g}, e^{- \hat{\phi}} dV_{\hat{g}},  m )$.
\end{defi}

\begin{rem}
We will denote with a hat all quantities computed with respect to the  smooth metric measure space $(M^n, \hat{g}, e^{- \hat{\phi}} dV_{\hat{g}},  m ) $. On the other hand, equalities in \eqref{conformal relation} imply   
\[
e^{- \hat{\phi}} dV_{\hat{g}} = e^{\frac{m + n }{m + n - 2} \sigma} e^{- \phi} dV_g.
\]
\end{rem}

\begin{defi} Given a compact smooth metric measure space $(M^n, g, e^{- \phi} dV_g,  m )$, the weighted Laplacian $\Delta_{\phi}: C^{\infty}(M) \to C^{\infty}(M)$ is an operator defined by
\[
\Delta_{\phi} u = \Delta_g  u  - \nabla_g  u  \cdot \nabla_g  \phi,
\] where $u \in C^{\infty}(M)$.
\end{defi}

\begin{defi}
Let $(M^n, g, e^{- \phi} dV_g,  m )$ be a smooth metric measure space. The weighted conformal Laplacian $L^m_{\phi}$ is given by the operator 
\begin{equation*}
L^m_{\phi} = - \Delta_{\phi	} + \dfrac{m+ n - 2}{4(m + n - 1)} R^{m}_{\phi}.
\end{equation*}
\end{defi}

\begin{pro}\label{conformal laplacian}
Let $(M^n, g, e^{- \phi} dV_g, m)$ and $(M^n, \hat{g}, e^{- \hat{\phi}} dV_{\hat{g}},  m ) $ be two pointwise conformally equivalent smooth metric measure spaces such that \small $\hat{g} = e^{\frac{2 \sigma}{m + n - 2}}g$ \normalsize and $\hat{\phi} = \frac{-m\sigma}{m+n-2} + \phi$. Let us denote by $L^m_{\phi}$ and $\hat{L}^m_{\hat{\phi}}$ their respective weighted conformal Laplacians. Then, we have $\hat{v} =  e^{\frac{\sigma}{m+n-2}}v$, and the following transformation rules
\begin{equation}\label{cambio en laplaciano conforme}
\hat{L}^m_{\hat{\phi}} (w) = e^{-\frac{m + n  + 2}{2(m + n - 2)} \sigma} L^m_{\phi} (e^{ \frac{\sigma}{2}} w),
\end{equation}
\begin{equation}\label{ecuacion conforme2}
\nabla_{\hat{g}} \hat{v } = e^{-\frac{\sigma}{m+n-2}} \left(  \nabla_{g} v  + \frac{v}{m+n-2} \nabla_{g} \sigma \right),
\end{equation}
\begin{equation}\label{escalar}
R_{\hat{g}} = e^{-\frac{2\hat{\sigma}}{m+n-2}}(R_{g} + \frac{2(n-1)}{m+n-2} \Delta_{g} \sigma -  \frac{(n-1)(n-2)}{(m+n-2)^2} |\nabla_{g} \sigma|_{g}^{2}), 
\end{equation} and  
\begin{equation}\label{Ricci}
\begin{array}{ll}
\hat{R}_{ij}(p) &= R_{ij} - \frac{n-2}{m+n-2} \sigma_{ij} + \frac{n-2}{(m+n-2)^2} \sigma_{i} \sigma_{j} + \left( \frac{\Delta_{g} \sigma}{(m+n-2)} -  \frac{n-2}{(m+n-2)^2} |\nabla_{g} \sigma|_{g}^{2}  \right) g_{ij}.
\end{array}
\end{equation}
\end{pro}

\begin{proof}
We mention that the identity \eqref{cambio en laplaciano conforme} appears in \cite{CaseYamabe}. Equality \eqref{escalar} follows from \eqref{cambio en laplaciano conforme} when $m=0$ and $w \equiv 1$. The identity $\hat{v} =  e^{\frac{\sigma}{m+n-2}}v$ follows from the relations $\hat{v}^{m} = e^{-\hat{\phi}}$, $v^{m} = e^{-\phi}$, and $\hat{\phi} = \frac{-m\sigma}{m+n-2} + \phi$.  
The equation \eqref{ecuacion conforme2} follows from equalities $\hat{v} =  e^{\frac{\sigma}{m+n-2}}v$ and $\nabla_{\hat{g}} = e^{-\frac{2\sigma}{m + n - 2} } \nabla_{g}$. For the equation \eqref{Ricci}, see (1.1) in \cite{EscobarAn}.
\end{proof}


 We denote by $(w, \varphi)_{M} = \int_{M} w \varphi \, v^{m} dV_g$ the inner product in $L^{2}(M, v^{m} dV_g)$. Additionally, we denote by $||\cdot||_{2,M}$ the norm in the space $L^{2}(M, v^{m} dV_g)$; in some cases, we use the notation $||\cdot||$ for this norm. $H^1(M, v^{m} dV_g)$ denotes the closure of $C^{\infty} (M)$ with respect to the norm $\int_{M}|\nabla w|_g^{2} + |w|_g^{2}$. Here and subsequently, the integrals are computed using the measure \small $v^{m} dV_g$. \normalsize

\section{Yamabe-type problem}\label{Yamabe}

In this section, we recall some concepts necessary to study the Yamabe-type problem in a smooth measure space $(M^n, g, v^m dV_g)$. In particular, we consider the weighted Yamabe quotient, which generalizes the Sobolev quotient in the case of $m=0$ and a suitable $\mathcal{W}$-functional. These definitions are taken from \cite{CaseYamabe}. Following the presentation in \cite{CaseYamabe}, we also consider the energies of these functionals and some of their properties. 


\subsection{The weighted Yamabe quotient}

We start with the definition of the weighted Yamabe quotient.

\begin{defi}

The \textit{weighted Yamabe quotient} $\mathcal{Q}[M^n, g, v^{m} dV_g]: C^{\infty} (M) \to \R$ of a compact smooth metric measure space $(M^n, g, e^{- \phi} dV_g,  m )$ is, by definition, the functional

\begin{equation*}\label{Yamabe funcional tilda}
\begin{array}{ll}
\mathcal{Q}[M^n, g, v^{m} dV_g](w) = & \dfrac{(L^{m}_{\phi} w, w)_M (\int_{M}| w| ^{\frac{2(m + n - 1)}{m +n - 2}}v^{-1}) ^{\frac{2m }{n}}}{ (\int_{M } |w|^{\frac{2(m + n)}{m +n -2}} )^{\frac{2m + n -2}{n}}}.
\end{array}  
\end{equation*}


The \textit{weighted Yamabe constant} $\Lambda[M^n, g, v^{m} dV_g] \in \R$ of  $(M^n, g, v^{m} dV_g )$ is defined by

\begin{equation*}
\Lambda[M^n, g, v^{m} dV_g] = \inf \{  \mathcal{Q}[M^n, g, v^{m} dV_g](w) : w \in H^1(M, v^{m} dV_g) \setminus \{ 0\} \}.
\end{equation*}

\end{defi}

The \textit{weighted Yamabe problem} is to find a function that minimizes the weighted Yamabe quotient.
  
\begin{rem}\label{simple notacion}
In some cases, when the context is clear, we will not write the dependence of the smooth metric measure space: for example, we write 
$\mathcal{Q}$ and $\Lambda$ instead of $\mathcal{Q}[M^n, g, v^{m} dV_g]$ and $\Lambda[M^n, g, v^{m} dV_g]$, respectively. We note that since $C^{\infty} (M)$ is dense in $H^1(M, v^{m} dV_g)$ and $\mathcal{Q}(|w|) = \mathcal{Q}(w)$, it is sufficient to consider the weighted Yamabe constant by minimizing over the space of nonnegative smooth functions on $M$, and we will subsequently make this assumption without further comment. 
\end{rem}





Note that the weighted Yamabe quotient is conformally invariant in the sense of definition \ref{conformal definition} (see Proposition 3.3 in \cite{CaseYamabe}). To simplify the computations and to avoid the trivial noncompactness of the weighted Yamabe problem, as in \cite{CaseYamabe}, we consider the next definition:

\begin{defi}\label{interior normalized}
Let $(M^n, g, v^{m} dV_g)$ be a smooth metric measure space. A smooth positive function $w$ is volume-normalized if

\begin{equation*}
\displaystyle \int_{M} w^{\frac{2(m+n)}{m+n-2}}= 1.
\end{equation*}

\end{defi}

\subsection{ $\mathcal{W}$-functional}

Let us start with the definition of the $ \mathcal{W}$-functional considered by Case in \cite{CaseYamabe}.

\begin{defi}
The functional $\mathcal{W}: C^{\infty}(M) \times \R^{+} \to \R$ of a compact smooth metric measure space $(M^n, g, v^{m} dV_g)$ is defined by

\[
\mathcal{W}(w, \tau) = \tau^{\frac{m}{m+n}} (L^{m}_{\phi} w, w)  +  m\int_{M} \left( \tau^{-\frac{n}{2(m+n)}} w^{\frac{2(m+n-1)}{m+n-2}}v^{-1} -   w^{\frac{2(m + n)}{m+n-2}} \right)
\] when $m \in [0, \infty)$. We will refer to this functional as the $\mathcal{W}$-functional.
\end{defi}

Additionally, $\mathcal{W}$ satisfies the following conformal property; see Proposition 3.10 in \cite{CaseYamabe}.

\begin{pro}[\cite{CaseYamabe}]\label{weighted Yamabe energy conformal tilda}
Let $(M^n, g, v^{m} dV_g)$ be a compact smooth metric measure space. In the first component, the $\mathcal{W}-$functional is conformally invariant:

\begin{equation}\label{W conformal Case}
\mathcal{W}[M^n, e^{2 \sigma} g,  e^{(m+n)\sigma}v^{m} dV_g](w, \tau) =
\mathcal{W}[M^n, g, v^{m} dV_g](e^{\frac{(m+n -2)}{2}\sigma} w, \tau)
\end{equation} for all  $\tau > 0$, $\sigma$, and $w \in C^{\infty}(M)$. 

\end{pro}

Since we are interested in minimizing the weighted Yamabe quotient, it is natural to define the following energies as infima of the $\mathcal{W}$-functional. It is also natural to relate one of these energies with the weighted Yamabe constant.

\begin{defi}
Given a compact smooth metric measure space $(M^n, g, v^{m} dV_g)$ and $\tau > 0$, the $\tau$-energy, denoted by $\nu[M^n, g, v^{m} dV_g](\tau) \in  \R$, is defined to be
\[
\nu[M^n, g, v^{m} dV_g](\tau) = \inf \left\{ \mathcal{W} (w, \tau) : w \in H^1(M, v^{m} dV_g),  \int_{ M} w ^{\frac{2(m+n)}{m+n-2}} = 1 \right\}.
\]

Define the energy $\nu[M^n, g, v^{m} dV_g] \in  \R \cup \{ - \infty \}$, where

\[
\nu[M^n, g, v^{m} dV_g] = \inf_{\tau >0} \nu[M^n, g, v^{m} dV_g](\tau).
\]
\end{defi}

The conformal invariance property in the $\mathcal{W}$-functional is transferred to the energies.

\begin{pro}[\cite{CaseYamabe}]\label{properties}
Let $(M^n, g, v^{m} dV_g)$ be a compact smooth metric measure space. Then

\begin{equation*}
\nu[M^n, g, v^{m} dV_g](\tau) = \nu[M^n, ce^{2 \sigma} g,  e^{(m+n)\sigma}v^{m} dV_{cg}] (c\tau),
\end{equation*}

\begin{equation*}
\nu[M^n, g, v^{m} dV_g] = \nu[M^n, ce^{2 \sigma} g,  e^{(m+n)\sigma}v^{m} dV_{cg}]
\end{equation*} for all $c > 0$, and for all $\sigma \in C^{\infty}(M)$.
\end{pro}

The following proposition shows that it is equivalent to considering the energy instead of the weighted Yamabe constant when the latter is nonnegative. 

\begin{pro}[\cite{CaseYamabe}]\label{W y Yamabe tilda}
Let $(M^n, g, v^{m} dV_g)$ be a compact smooth metric measure space, and let $\nu$ and $\Lambda$ denote the energy and the weighted Yamabe constant, respectively. Then

\begin{itemize}
\item[(a)]  $\nu = -\infty$ if and only if $\Lambda \in [-\infty, 0)$;
	
 \item[(b)]  $\nu = -m$ if and only if $\Lambda = 0$; and
 \item[(c)]  $\nu > -m$ if and only if $\Lambda > 0$. Moreover, in this case we obtain

\begin{equation*}
	\nu = \frac{2m + n }{2} \left[\frac{2 \Lambda}{n}  \right]^{\frac{n}{2m+n}} - m
\end{equation*} and if $w$ is a volume-normalized function, we have that $(w, \tau)$ is a minimizer of $\nu$ with

\begin{equation*}\label{nu lambda}
	\tau = \left[ \dfrac{n \int_{M}  w^{\frac{2(m+n-1)}{m+n-2}}v^{-1} }{2 (L^{m}_{\phi} w, w)} \right]^{\frac{2(m+n)}{2m + n}}.
\end{equation*} if and only if $w$ is a minimizer of $\Lambda$.

\end{itemize}

\end{pro}

Next, we consider the Euler-Lagrange equation of the $\mathcal{W}$-functional.

\begin{pro}[\cite{CaseYamabe}]\label{funcion realiza nu tilda}
Let $(M^n, g, v^{m} dV_g)$ be a compact smooth metric measure space $(M^n, g, v^{m} dV_g)$. Fix $\tau > 0$ and consider the map $\xi \to \mathcal{W}(\xi, \tau)$, where every $\xi$ is a volume-normalized function in $H^{1}(M, v^m dV_g)$. Suppose that $w \in H^{1}(M, v^m dV_g)$ is a critical point of this map. Then, $w$ is a weak solution of

\begin{equation}\label{ecuacion tau tilda}
\begin{array}{c}
\tau^{\frac{m}{m+n}}  L^{m}_{\phi}w + \frac{m(m+n-1)}{m+n-2} \tau^{-\frac{n}{2(m + n)}}  w^{\frac{m+n}{m+n-2}}v^{-1} = c w^{\frac{m+n+2}{m+n-2}}, \\
\end{array}
\end{equation} for some constant $c$. Furthermore, if $(w, \tau)$ minimizes the $\nu$-energy, then

\begin{equation*}\label{ecuacion tau tilda c}
c = \frac{(2m + n - 2)(m + n)}{(2m + n)(m + n - 2)} (\nu + m).
\end{equation*}

\end{pro}




\subsection{\bf Euclidean space as the model space for the weighted Yamabe problem}

In this subsection, we consider a family of functions together with some of its properties, which are fundamental in our proof of the Aubin-type existence result for minimizers of the Yamabe quotient. 

We start by mentioning Del Pino and Dolbeault's result in \cite{DelPino} regarding the sharp Gagliardo-Nirenberg-Sobolev inequalities in the manner as Case presented it in \cite{CaseYamabe}.

\begin{teo}[\cite{DelPino}]\label{teoGN} Fix $m \in [0, \infty)$. For all $w \in H^{1}(\R^{n}) \cap L^{\frac{2(m+n-1)}{m+n-2}} (\R^{n})$, it holds that

\begin{equation}\label{GNS}\Lambda_{m, n} \left( \int_{\R^{n}} | w| ^{\frac{2(m + n)}{m +n - 2}} \right) ^{\frac{2m + n -2}{n}}   \leq \left( \int_{\R^{n}} |\nabla w|^{2}  \right) \left( \int_{\R^{n}} |w|^{\frac{2(m + n - 1)}{m + n -2}} \right)^{\frac{2m }{n}} \end{equation}, where the constant $\Lambda_{n, m}$ is given by

\begin{equation*}\label{c GNS}\Lambda_{m,n} =   \dfrac{n\pi(m+n-2)^2}{(2m + n - 2)}   \left(\dfrac{2(m+n-1)}{(2m + n -2)} \right)^{\frac{2m}{n}} \left( \dfrac{ \Gamma(\frac{2m + n}{2})}{\Gamma(m + n)} \right)^{\frac{2}{n}}.\end{equation*}

Moreover, the equality holds in \eqref{GNS} if and only if $w$ is a constant multiple of the function $w _{\epsilon, x_0}$ defined on $\R^{n}$ by

\begin{equation}\label{funcion Del Pino}w _{\epsilon, x_0} (x) := \left(\frac{2 \epsilon}{\epsilon^2 + |x - x_0|^2}  \right)^{\frac{m+n-2 }{2}}\end{equation} where $\epsilon > 0 $, and $x_0 \in \R^{n}$. \end{teo}

Fix $n\geq 3$ and $m\geq 0$. Next, we consider a particular case of function in \eqref{funcion Del Pino} defined by 

\begin{equation}\label{burbuja Case}
\varphi_{x_0, \tau}(x) = \tau^{-\frac{n(m+n-2)}{4(m+n)}}\left(1 + \frac{c(m,n)}{\tau}|x - x_0|^2 \right)^{-\frac{(m+n-2)}{2}}.
\end{equation} where $c(m,n) = \frac{m+n-1}{(m+n-2)^2}$, $x_0 \in \R^n $ and $\tau > 0$.

We denote the normalization of $\varphi_{x_0, \tau}$ by

\[
\tilde{\varphi}_{x_0, \tau} = V^{-\frac{m+n-2}{2(m+n)}} \varphi_{x_0, \tau},
\] where

\begin{equation}\label{Definition V}
V = \int_{\R^n} \varphi_{0, 1}^{\frac{2(m+n)}{m+n-2}} dx = \int_{\R^n} \varphi_{x_0, \tau}^{\frac{2(m+n)}{m+n-2}} dx;
\end{equation} we used the change of variables in the second equality. On the other hand, a computation shows

\begin{equation}\label{ecuacion burbujas case}
- \tau^{\frac{m}{m+n}} \Delta \varphi_{x_0, \tau} + \frac{m(m+n-1)}{m+n-2}  \frac{\varphi_{x_0, \tau}^{\frac{m+n}{m+n-2}}}{\tau^{\frac{n}{2(m + n)}}} = \frac{(m+n)(m+n-1)}{m+n-2} \varphi_{x_0, \tau}^{\frac{m+n+2}{m+n-2}}.
\end{equation}

The definition of $\varphi_{x_0,\tau}$, the definition of $V$, and identity \eqref{ecuacion burbujas case} correspond to (5.1), (5.2), and (5.3) in \cite{CaseYamabe}, respectively. Since $\tilde{\varphi}_{x_0, \tau}$ is a volume-normalized function and attains the infimum of the weighted Yamabe quotient (see Theorem \ref{teoGN}), by Proposition \ref{W y Yamabe tilda} there is $\tilde{\tau} > 0$ such that 

\begin{equation}\label{nu Rn case}
\begin{array}{rl}
\nu(\R^{n}, dx^{2}, 1^{m} dV_g)  + m & = \mathcal{W}(\R^{n},  dx^{2}, 1^{m} dV_g)(\tilde{\varphi}_{x_0, \tau},  \tilde{\tau}) + m \vspace{0.2cm}\\
& = \frac{\tilde{\tau}^{\frac{m}{m+n}}} {V^{\frac{m+n-2}{m+n}} } \int_{\R^{n}} |\nabla \varphi_{x_0, \tau}|^{2} + \frac{m\tilde{\tau}^{-\frac{n}{2(m + n)}}}{ V^{\frac{m+n-1}{m+n}}} \int_{\R^{n}}  \varphi_{x_0, \tau}^{\frac{2(m+n-1)}{m+n-2}}.
\end{array}
\end{equation}  

It follows that $\tilde{\tau} = \tau V^{-\frac{2}{2m+n}}$, since $\tilde{\varphi}_{x_0, \tau}$ satisfies the equation \eqref{ecuacion tau tilda} and $\varphi_{x_0, \tau}$ satisfies the equation \eqref{ecuacion burbujas case}, respectively.

The next result, which corresponds to Theorem 7.1 in \cite{CaseYamabe}, links the weighted Yamabe constants of $(M^n, g, v^m dV_g)$ and $(M^n, g, v^{m+1} dV_g)$ with the weighted Yamabe constants for the Euclidean space with parameters $m$ and $m+1$. This result allows us to prove the existence of a minimizer for the weighted Yamabe constant in an inductive argument for the parameter $m$.

\begin{teo}[\cite{CaseYamabe}]\label{inductive constants}

Let $(M^n, g, v^{m} dV_g)$ be a compact smooth metric measure space with a nonnegative weighted Yamabe constant, and suppose that the weighted Yamabe constant was minimized by a smooth positive function. Then

\begin{equation*}
\Lambda[M^n, g, v^{m+1} dV_g] \leq \Lambda[\R^n, dx^2, dV, m+1] \dfrac{\Lambda[M^n, g, v^{m} dV_g]}{\Lambda[\R^n, dx^2, dV, m]}.
\end{equation*}
\end{teo}

\section{An Aubin-type existence theorem}\label{Aubin}

In this section, we are dedicated to proving Theorem A. Briefly, the proof consists of performing some computation with a family of test functions in the $\mathcal{W}$-functional around a point $p$ where the Weyl tensor is nonzero. To simplify these computations, we consider a smooth metric measure space conformal to the original so that $T(p) = 0$ and the new density $v$ has suitable properties. The family of functions are supported in a small neighborhood of the point $p$. Such functions are of the form of a standard cutoff function times the family of functions $\varphi_{x_0, \tau}$ defined by \eqref{burbuja Case}. Then,  taking the limit when the parameter $\tau$ goes to zero, we prove that the entropy is less than the entropy of Euclidean space when $m < \delta$, for some $\delta >0$. By Proposition \ref{W y Yamabe tilda} and Theorem \ref{Aubin Case}, we obtain the result for $m < \delta$. Finally,  using Theorem \ref{inductive constants} in an inductive argument, we obtain the result for $m \in  \bigcup \limits_{i \in \mathbb{N} \cup \{ 0 \}} [i,i + \delta)$. 

In this section, on the right-hand side of every equality or inequality that involves the terms $R^{m}_{\phi}$, $v^k$, $R_{g}$, $R_{ij}$, $T_{ij}$ $R_{ijkl}$ or $W_{ijkl}$,  will be calculated in the point $p$ and we will omit this point from notation. We will denote $ \Delta^2 f $ instead $\Delta(\Delta f)$ where $f \in C^{\infty}(M)$. The big O notation is considered when the variables involved are going to zero. Additionally, in this section, $C$ is a positive constant that depends only on the smooth metric measure $(M^n,  g, v^{m} dV_g)$ and possibly changes from line to line or within the same line. Finally, we use the Einstein summation convention when an index variable appears twice in a single term. In some few cases, we will write the summation notation when the indexes appear more than twice or when the index is not over all of the values.


\subsection{Suitable smooth measure space} In this subsection, we consider a suitable conformally smooth measure space to begin the proof of Theorem A.

\begin{lema}\label{cambio conforme especial}
Let $(M^n, g, v^{m} dV_g)$ be a compact smooth metric measure space. There exists a conformal smooth metric measure space $(M^n,  \hat{g}, \hat{v}^{m} dV_{\hat{g}})$ such that for $p \in M$ we have $\hat{v}(p) = 1$, $\nabla_{\hat{g}} \hat{v}(p) = 0$, and $Ric_{\hat{g}}(p) = 0$.  
\end{lema}

\begin{proof}
If $\sigma \in  C^{\infty}(M)$, $\tilde{g} = e^{\frac{2\sigma}{m+n-2}}g$, by equations \eqref{ecuacion conforme2} and \eqref{escalar} we can take $\sigma$ and $\tilde{v} = e^{\frac{2\sigma}{m+n-2}}v$, such that we have $\tilde{v}(p) = 1$, $\nabla_{\tilde{g}} \tilde{v}(p) = 0$ and $R_{\tilde{g}}(p) =0$. Then, we change our original smooth metric measure space by $(M^n,  \tilde{g}, \tilde{v}^{m} dV_{\tilde{g}})$. We denote with tilde all quantities computed with respect to the smooth metric measure space $(M^n, \tilde{g}, \tilde{v}^{m} dV_{\tilde{g}}) $.

 Let $\hat{g} = e^{\frac{2\hat{\sigma}}{m+n-2}}\tilde{g}$ and $\hat{v} =  e^{\frac{\hat{\sigma}}{m+n-2}}\tilde{v}$ be such that in the point $p$, $\hat{\sigma}$ satisfies that $\hat{\sigma} = 0$, $\nabla_{\tilde{g}} \hat{\sigma} = 0$, and $ \hat{\sigma}_{ij} = \frac{m+n-2}{n-2}\tilde{T}_{ij} = \frac{m+n-2}{n-2}\tilde{R}_{ij}$. Since $\tilde{T}_{ij}$ is trace-free, we obtain in the point $p$ that $\Delta_{\tilde{g}} \hat{\sigma} = 0$. We denote by a hat all quantities computed with respect to the metric $\hat{g}$. In $p$, the relation \eqref{Ricci} for the Ricci tensor in the metrics $\hat{g}$ and $\tilde{g}$ yields $\hat{R}_{ij}(p) = 0$. Given that in $p$ we have $\nabla_{\hat{g}} \tilde{v} = 0$, $\nabla_{\hat{g}} \hat{\sigma} = 0$, and transformation rule \eqref{ecuacion conforme2} yields $\nabla_{\hat{g}} \hat{v }(p) = 0$. The smooth metric measure space requested is $(M^n,  \hat{g}, \hat{v}^{m} dV_{\hat{g}})$.  
\end{proof}

In the following lemma, we present some computation for a smooth metric measure space which satisfies the thesis of Lemma \ref{cambio conforme especial}.

\begin{lema}\label{Calculus v}
Let $(M^n, g, v^{m} dV_g)$ be a compact smooth metric measure space such that for $p \in M$ we have $v(p) = 1$, $\nabla_{g} v(p) = 0$, and $Ric_{g}(p) = 0$. Then, for $k \in \R$, we have


\begin{equation}\label{Laplacian vk}
  \Delta_g v^k(p) = k \Delta_g v(p),
\end{equation}

\begin{equation}\label{doble laplaciano vm}
\begin{array}{rl}
    \Delta_{g}^2 v^k (p) =  & (v^k)_{ijij} \\
    = &  2k(k-1) |\text{Hess} \, v|_{g}^2 + k(k-1)(\Delta v)^2 + k \Delta_{g}^2 v + k(k-1) (\Delta_{g} v)^2,
    \end{array}
\end{equation} and

\begin{equation}\label{Delta Rm}
    \Delta_{g} R^m_{\phi} (p) = \Delta_{g} R_{g} + 2m(1-m) |\text{Hess} \, v|_{g}^2 - 2m\Delta^2 v + 2m (\Delta v )^2.
\end{equation}

\end{lema}

\begin{proof}

We will consider normal coordinates around $p$ to compute the covariant derivatives. By hypothesis $v(p) = 1$ and $\nabla_{g} v(p) = 0$, we obtain

\begin{equation*}
   \Delta_g v^k(p) = k\Delta_gv + k(k-1)|\nabla_g v|^{2}_{g} = k\Delta_gv.
\end{equation*}

The Ricci formula implies  

\begin{equation}\label{vijj}
(v^k)_{ijj}(p) = (v^k)_{jij} = (v^k)_{jji} + R^{s}_{jij}(v^k)_{s} = (v^k)_{jji} + R_{sj} (v^k)_{s} = (v^k)_{jji}
\end{equation} and \small
\begin{equation*}
(v^k)_{ijij}(p) = ((v^k)_{iij})_j = ((v^k)_{iij} + R^{s}_{iji}(v^k)_s)_j = (v^k)_{iijj} + R^{s}_{iji;j} (v^k)_s + R^{s}_{iji}(v^k)_{sj}= \Delta^2 v^k.
\end{equation*} 
\normalsize

To compute the second equality in \eqref{doble laplaciano vm},  
we will use $v(p) = 1$, $\nabla_{g} v(p) = 0$, and \eqref{vijj} to obtain \small
\begin{equation}\label{Delta 2}
\begin{array}{rl}
\Delta_{g}^2 v^k (p) = &k(k-1)(k-2)(k-3)|\nabla_{g}v|_{g}^4 + 2k(k-1)(k-2)|\nabla_{g}v|^2 \Delta_{g}v \\ 
& + 4k(k-1)(k-2) \text{Hess} \,v(\nabla_{g}v, \nabla_{g}v) 
+ 2k(k-1) |\text{Hess} \, v|_{g}^2 + k(k-1) (\Delta_{g} v)^2 \\
& + 4k(k-1) \langle \nabla v, \nabla \Delta v \rangle  + k \Delta_{g}^2 v + 2k(k-1) Ric_{g}(\nabla v, \nabla v)\\
=&  2k(k-1) |\text{Hess} \, v|_{g}^2 + k \Delta_{g}^2 v + k(k-1) (\Delta_{g} v)^2.\\
\end{array}
\end{equation}

\normalsize

Finally, by hypothesis, Ricci formula, and formula \eqref{Bakryv} we obtain \small
\begin{equation}
\begin{array}{rl}
\Delta_{g} R^m_{\tilde{\phi}} (p) = &\Delta_{g} (R_{g} + \frac{m(1-m)}{v^2} | \nabla_{g} v|^2 - \frac{2m}{v}\Delta_{g} v )\\
= & \Delta_{g} R_{g}   + 2m(m-3)|\nabla_{g}v|^2 \Delta_{g}v + 2m (\Delta_{g} v)^2 \\ 
&-2m(m-3) \langle \nabla_{g} v, \nabla_{g} \Delta_{g} v \rangle - 2 m \Delta_{g}^2 v + 6m(1-m) |\nabla_{g}v|_{g}^4 \\
&+ 8m(m-1) \text{Hess} \,v(\nabla_{g}v, \nabla_{g}v) + 2m(1-m) |\text{Hess} \, v|_{g}^2 \\
& + 2m(1-m) Ric_{g}(\nabla_{g} v, \nabla_{g} v) \\
= & \Delta_{g} R_{g} + 2m(1-m) |\text{Hess} \, v|_{g}^2 - 2 m \Delta_{g}^2 v + 2m (\Delta_{g} v)^2.
\end{array}
\end{equation} \end{proof}
\normalsize
Next, we write some formulae needed in the proof of Theorem A. In $g$-normal coordinates, it holds that
\begin{equation}\label{dV normal}
\begin{array}{ll}
dV_{g}  = &(1 - \frac{1}{6}R_{ij}x_{i} x_{j} - \frac{1}{12}R_{ij,k}x_{i} x_{j} x_{k} \\
& \, \, - ( \frac{1}{40}R_{ij,kl} + \frac{1}{180}R_{pijr} R_{pklr} - \frac{1}{72}R_{ij} R_{kl})x_{i} x_{j} x_{k} x_{l} + O(|x|^5) )dx
\end{array} %
\end{equation} and
\begin{equation}\label{g -1}
\begin{array}{rl}
g^{ij}(x)  =  \delta_{ij} - \frac{1}{3}R_{iklj} x_{k} x_{l} - \frac{1}{6} R_{iklj,s}  x_{k} x_{l} x_{s} 
- (\frac{1}{20} R_{iklj,su} - \frac{3}{45} R_{iklr}R_{jsur} )
x_{k} x_{l} x_{s} x_{u}  + O(|x|^5),
\end{array}
\end{equation} where $1 \leq i, j, k,l,r, s, u \leq n$, we recall that the coefficients are computed in $p$. For formula \eqref{dV normal}, see Lemma 5.5 in \cite{LeeParker}, and for formula \eqref{g -1}, see formula (5.4) in \cite{LeeParker} and an argument akin to Lemma 2.3 in \cite{Coda2005}.

We recall the following identities
\begin{equation}\label{T}
\begin{array}{l}
T_{ij}  = R_{ij} - \dfrac{R}{n}g_{ij},
\end{array}
\end{equation}
\begin{equation}\label{W}
\begin{array}{l}
R_{ijkl} R^{ijkl} = W_{ijkl}W^{ijkl} + \dfrac{4}{n-2}T_{ij} T^{ij} + \dfrac{2R_{g}^2}{n(n-1)}.
\end{array}
\end{equation}

The identity \eqref{W} is taken from page 272 in \cite{Aubin}.

\vspace{0.3cm}

\subsection{Some estimates}\label{local expansion}  

The underlying idea of the proof of Theorem A is to improve the upper bound estimated in Proposition 6.3 in \cite{CaseYamabe}. For this purpose, we fix a point $p \in  M$ and let $\{x_i \}$ be normal coordinates in some fixed neighborhood $U$, centered at $p := (0, ..., 0)$. Let $0< \epsilon < 1$ be such that $B_{2 \epsilon}(p) \subset U$, where $B_{2 \epsilon}(p)$ is the geodesic ball of radius $2\epsilon$ around $p$. Let $\eta : M \to [0, 1]$ be a cutoff function such
that $\eta \equiv 1$ on $B_{\epsilon}$, $\mathrm{supp} (\eta) \subset B_{2\epsilon}$ and $|\nabla \eta| < C\epsilon^{-1}$ in $A_{\epsilon} := B_{2\epsilon} \smallsetminus B_{\epsilon}$. For each $0 < \tau < 1$, define $f_{\tau} : M \to \R$ by 
\begin{equation}\label{f tau}
f_{\tau} (x_1, . . . , x_{n}) = \eta \varphi_{0,\tau} (x_1, . . . , x_{n}), \end{equation} and set  
\begin{equation}\label{f tau tilde}
\tilde{f}_{\tau} = V_{\tau}^{-\frac{m+n-2}{2(m+n)}} f_{\tau} \end{equation} where
\begin{equation}\label{definition Vtau}
V_{\tau} = \int_{M} f_{\tau}^{\frac{2(m+n)}{m+n-2}} v^m dV_{g}.
\end{equation}

Taking $\tilde{\tau} = \tau V^{-\frac{2}{2m+n}}$, where $V$ is defined by \eqref{Definition V}, the definition of $\mathcal{W}$ and equality \eqref{W conformal Case} in Proposition \ref{weighted Yamabe energy conformal tilda} yield
\begin{equation}\label{w8}
\begin{array}{rl}
\mathcal{W}[M^n,  g, v^{m} dV_{g}](\tilde{f}_{\tau}, \tilde{\tau}) + m 
= & \displaystyle \dfrac{\tilde{\tau}^{\frac{m}{m+n}}}{ V_{\tau}^{\frac{m+n-2}{m+n}} }\left (\int_{B_{2\epsilon}} |\nabla_g f_{\tau}|_{g}^{2} v^{m} 
+ \frac{m+n-2}{4(m+n-1)}R^m_{\phi} f_{\tau}^2 v^{m} dV_{g}  \right) \vspace{0.2cm} \\
& + \displaystyle \dfrac{m\tilde{\tau}^{-\frac{n}{2(m+n)}} }{ V_{\tau}^{\frac{m+n-1}{m+n}} } \int_{B_{2\epsilon}} f_{\tau}^{\frac{2(m+n-1)}{m+n-2}}v^{m-1}	 dV_{g}.
\end{array}
\end{equation}

Next, we establish lemmas to estimate each term on the right-hand side of the equality \eqref{w8}. First, we estimate on the right-hand side of \eqref{w8} the term with the Bakry-\'Emery scalar curvature  $R^{m}_{\phi}$.

\begin{lema}\label{Estimate R}
Let $(M^n, g, v^{m} dV_g)$ be a compact smooth metric measure space such that for $p \in M$, we have $v(p) = 1$, $\nabla_{g} v(p) = 0$, and $Ric_g(p) = 0$. If $f_{\tau}$ is defined by \eqref{f tau}, then, 
\small
\begin{equation}\label{R5 Case}
\begin{array}{ll} %
\displaystyle \int_{B_{2\epsilon} } R^{m}_{\phi} f_{\tau}^2 v^m dV_{g}
 = & - 2m\Delta_g v\frac{\tau^{\frac{n}{m+n}}}{c(m,n)^{\frac{n}{2}}} I_1   + ( \Delta_g  R^{m}_{\phi}  -2m^2  (\Delta_g v)^2 ) \dfrac{\tau^{\frac{n}{m+n} + 1} I_2 }{2nc(m,n)^{\frac{n+2}{2}}} \vspace{0.2cm} \\
&+  O(\tau^{\frac{n}{m+n}+1 + \frac{q}{2}} ) + O(\epsilon^{4 -2m-n} \tau^{\frac{n}{m+n} + m + \frac{n-4}{2}}),
\end{array}
\end{equation} \normalsize where 
\begin{equation}\label{I1 Case}
I_1 = \int_{\R^{n}} (1 + |y|^2)^{-(m+n-2)} dy 
\end{equation} and
\begin{equation}\label{I6 Case}
I_2 = \int_{\R^{n}} |y|^2(1 + |y|^2)^{-(m+n-2)} dy.
\end{equation}
\end{lema}

\begin{proof} 
First, we estimate the term with the Bakry-\'Emery scalar curvature  $R^{m}_{\phi}$ in the region $A_{\epsilon}$. Given that $dV_{g} = (1+\epsilon C)dx$ around $p$, we obtain 
\begin{equation}\label{R1 Case}
\begin{array}{ll}
\displaystyle \int_{A_{\epsilon}}R^m_{\phi} f_{\tau}^2 v^{m} dV_{g} & \leq  C(1+ \epsilon C) \displaystyle \int_{A_{\epsilon}} \varphi_{0, \tau}^2 dx \vspace{0.2cm}\\
& = C(1+ \epsilon C) \tau^{-\frac{n(m+n-2)}{2(m+n)}} \displaystyle \int_{A_{\epsilon}} (1 + \frac{c(m,n)}{\tau}|x|^{2})^{-(m+n-2)} dx \vspace{0.2cm} \\
& = C(1+ \epsilon C) \tau^{\frac{n}{m+n}} \displaystyle \int_{A_{\frac{\epsilon\sqrt{c(m,n)}}{\sqrt{\tau}} }}  (1 + |y|^{2})^{-(m+n-2)} dy \vspace{0.2cm} \\
& = C(1+ \epsilon C) \tau^{\frac{n}{m+n}} \displaystyle \int_{\epsilon\tau^{-\frac{1}{2}}c(m,n)^{\frac{1}{2}}}^{2\epsilon\tau^{-\frac{1}{2}}c(m,n)^{\frac{1}{2}}}(1 + r^2)^{-(m+n-2)}r^{n-1}dr \vspace{0.2cm} \\
& \leq C(1+ \epsilon C) \epsilon^{4-2m-n} \tau^{\frac{n}{m+n} + m + \frac{n-4}{2}}.
\end{array}
\end{equation}

Next, we estimate on the right-hand side of \eqref{w8} the term with the Bakry-\'Emery curvature $R^{m}_{\phi}$ in the region $B_{\epsilon}$. For this purpose, we will use the Taylor expansion around $p$ of $R^{m}_{\phi}$ and $v^m$. Since $v(p) = 1$, $\nabla_g v(p) = 0$, $R_{ij}(p) =0$, and $R_g(p)=0$, it follows by formula \eqref{Bakryv} that $R^{m}_{\phi}(p) = - 2m\Delta_g v$, and the Taylor expansion of $R^{m}_{\phi}$ and $v^m$ takes the form of
\begin{equation}\label{Taylor R}
R^{m}_{\phi}(x) = - 2m\Delta_g v + (R^{m}_{\phi})_{i} x_{i} + \dfrac{(R^{m}_{\phi})_{ij}}{2} x_{i} x_{j} + O(|x|^3), 
\end{equation}
\begin{equation}\label{Taylor v}
v^m(x) = 1  + \frac{(v^{m})_{ij}}{2}x_{i}x_{j} + \frac{(v^{m})_{ijl}}{6}x_{i}x_{j}x_{l} + \frac{(v^{m})_{ijlk}}{24}x_{i}x_{j}x_{l}x_{k} + O(|x|^5),
\end{equation} we recall that the coefficients were computed in $p$. Since $R_{ij}(p) = 0$, formula \eqref{dV normal} takes the form 
\begin{equation}\label{dV normal 2}
\begin{array}{ll}
dV_{g}  = &(1 - \frac{1}{12}R_{ij,k}x_{i} x_{j} x_{k} - ( \frac{1}{40}R_{ij,kl} + \frac{1}{180}R_{pijr} R_{pklr} )x_{i} x_{j} x_{k} x_{l} + O(|x|^5) )dx.
\end{array} %
\end{equation}

Using equalities \eqref{Taylor R}, \eqref{Taylor v}, and \eqref{dV normal 2} and the symmetries in the ball, we have \footnotesize
\begin{equation}\label{R2 Case}
\begin{array}{rl}
\displaystyle  \int_{B_{\epsilon} } R^{m}_{\phi} f_{\tau}^2 v^m dV_{g}   = & \displaystyle  - 2m\Delta_g v \int_{B_{\epsilon}} \varphi_{0,\tau}^2  dx +  \frac{1}{2}( (R^{m}_{\phi})_{ij}  - 2m (v^m)_{ij} \Delta_g v) \int_{B_{\epsilon}} \varphi_{0,\tau}^2 x_{i}x_{j} dx \vspace{0.2cm} \\
& \displaystyle +  \int_{B_{\epsilon}} \varphi_{0,\tau}^2 O(|x|^{4})dx\\
=& \displaystyle   - 2m\Delta_g v \int_{B_{\epsilon}} \varphi_{0,\tau}^2  dx +  \frac{1}{2n}( \Delta_g  R^{m}_{\phi}  -2m  (\Delta_g v^m) \Delta_g v) \int_{B_{\epsilon}} \varphi_{0,\tau}^2 |x|^{2}dx \vspace{0.2cm} \\
& \displaystyle +  \int_{B_{\epsilon}} \varphi_{0,\tau}^2 O(|x|^{4})dx.\\
\end{array}
\end{equation}

\normalsize

Then, \begin{equation}\label{R3 Case}
\begin{array}{ll}
\displaystyle \int_{B_{\epsilon}} \varphi_{0,\tau}^2 dx & = \displaystyle\tau^{-\frac{n(m+n-2)}{2(m+n)}} \int_{B_{\epsilon}} (1 + \frac{c(m,n)}{\tau}|x|^2)^{-(m+n-2)}dx \vspace{0.2cm}\\
& =\displaystyle \frac{\tau^{-\frac{n(m+n-2)}{2(m+n)} + \frac{n}{2}}}{c(m,n)^{\frac{n}{2}}} \int_{B_{\frac{\epsilon\sqrt{c(m,n)}}{\sqrt{\tau}}}} (1 + |y|^2)^{-(m+n-2)} dy  \vspace{0.2cm} \\
& = \displaystyle \frac{\tau^{\frac{n}{m+n}}}{c(m,n)^{\frac{n}{2}}} (I_1 - \int_{ \R^{n} \setminus B_{\frac{\epsilon\sqrt{c(m,n)}}{\sqrt{\tau}}}} (1 + |y|^2)^{-(m+n-2)} dy) \vspace{0.2cm} \\
& =  \displaystyle \frac{\tau^{\frac{n}{m+n}}}{c(m,n)^{\frac{n}{2}}} I_1  + O(\epsilon^{4-2m-n}\tau^{\frac{n}{m+n} + m + \frac{n-4}{2}}) \\
\end{array}
\end{equation} and similarly,
\begin{equation}\label{R1 Case orden 2}
\begin{array}{ll}
\displaystyle \int_{B_{\epsilon}} \varphi_{0,\tau}^2 |x|^2 dx & = \displaystyle \tau^{-\frac{n(m+n-2)}{2(m+n)}} \int_{B_{\epsilon}} (1 + \frac{c(m,n)}{\tau}|x|^2)^{-(m+n-2)} |x|^2dx \vspace{0.2cm}\\
& = \frac{\tau^{\frac{n}{m+n} + 1}}{c(m,n)^{\frac{n+2}{2}}} I_2  + O(\epsilon^{6-2m-n}\tau^{\frac{n}{m+n} + m + \frac{n-4}{2}}). \\
\end{array}
\end{equation}

Now, taking $q < \min\{2m+n-6, 1\}$ and $0 < \epsilon < 1$, then for $|x| \leq \epsilon$ we obtain $|x|^{q} > |x|^{2}$ and 
\begin{equation}\label{R4 Case}
\begin{array}{ll}
\displaystyle \int_{B_{\epsilon}} \varphi_{0,\tau}^2 |x|^{4}dx & \leq \displaystyle \int_{B_{\epsilon}} \varphi_{0,\tau}^2 |x|^{2 + q}dx  \vspace{0.2cm}\\
& \leq \displaystyle \tau^{-\frac{n(m+n-2)}{2(m+n)}} \int_{B_{\epsilon}} \dfrac{|x|^{2+q} }{(1 + \frac{c(m,n)}{\tau}|x|^{2})^{m+n-2}}dx  \vspace{0.2cm}\\
& \leq  \displaystyle C\tau^{ \frac{n}{m+n} +1 + \frac{q}{2} } \int_{B_{\frac{\epsilon\sqrt{c(m,n)}}{\sqrt{\tau}}}} \dfrac{|y|^{2+q}}{(1 + |y|^2)^{(m+n-2)}}dy  \vspace{0.2cm} \\
& \leq C\tau^{ \frac{n}{m+n} +1 + \frac{q}{2}}(C + C\displaystyle \int_{1}^{\infty} r^{6-2m-n+q -1}dr)  \vspace{0.2cm}\\
& \leq C\tau^{ \frac{n}{m+n} +1 + \frac{q}{2}}. 
\end{array}
\end{equation}

Equality \eqref{Laplacian vk}, together with estimates \eqref{R1 Case}, \eqref{R2 Case}, \eqref{R3 Case}, and \eqref{R4 Case}, leads to equality \eqref{R5 Case}.
\end{proof}

The following lemma estimates the integral with the gradient term in \eqref{w8}

\begin{lema}\label{Estimate gradiente}
Let $(M^n, g, v^{m} dV_g)$ be a compact smooth metric measure space such that for $p \in M$, we have $v(p) = 1$, $\nabla_{g} v(p) = 0$, and $Ric_{g}(p) = 0$. If $f_{\tau}$ is defined by \eqref{f tau}, then
\begin{equation}\label{gradiente Case 6}
\begin{array}{rl}
\displaystyle \int_{B_{2\epsilon}} |\nabla_g f_{\tau}|_g^{2} v^{m} dV_{g}  = & \displaystyle  \int_{B_{\epsilon}} |\nabla \varphi_{0, \tau}|^{2} dx + \tau^{\frac{n}{m+n}} \frac{ m(m+n-2)^2}{2 n \, c(m,n)^{\frac{n}{2}}}   \Delta_{g}v I_3  \vspace{0.2cm} \\ 
& - \dfrac{\tau^{\frac{n}{m+n} + 1}(m+n-2)^2  }{2c(m,n)^{\frac{n+2}{2}} n(n+2)} I_4 A_1  + O( \tau^{\frac{n}{m+n} + m + \frac{n-4}{2}} \epsilon^{2-2m-n}) \vspace{0.2cm} \\
& + O(\tau^{\frac{n}{m+n} + 1  + \frac{q}{2}}),\\  
\end{array}
\end{equation} where $|\cdot|$ is the Euclidean norm,
\begin{equation}\label{I2 Case}
\begin{array}{ll}
I_3 = \displaystyle \int_{\R^{n}} |y|^4(1 + |y|^2)^{-(m+n)} dy,
\end{array}
\end{equation}
\begin{equation}\label{I7 Case}
\begin{array}{ll}
I_4 = \displaystyle \int_{\R^n } |y|^{6} (1+ |y|^2)^{-(m+n)}  dy, 
\end{array}
\end{equation} and
\begin{equation}\label{A1 2}
\begin{array}{ll}
A_1 = & \displaystyle  - \frac{m(m-1)}{2} |\text{Hess} \, v|_{g}^2 - \frac{m(m-1)}{4}  (\Delta_{g} v)^2 - \frac{m}{4} \Delta_{g}^2 v + \frac{\Delta R_{g}}{10} \displaystyle   + \frac{|W|_g^2}{60}. 
\end{array}
\end{equation}

\end{lema}

\begin{proof} To estimate the gradient term in $A_{\epsilon}$, note that, in $g$-normal coordinates, the term $|\nabla f_{\tau}|_{g}^{2}$ in $A_{\epsilon}$ satisfies the following inequality 
\begin{equation}\label{gradiente en el anillo}
|\nabla f_{\tau}|_{g}^{2} \leq C|\nabla f_{\tau}|^{2} \leq C (\eta^{2}|\nabla \varphi_{0, \tau}|^{2} + |\nabla \eta|^{2} \varphi_{0, \tau}^{2}) \leq C (|\nabla \varphi_{0, \tau}|^{2} + \epsilon^{-2} \varphi_{0, \tau}^{2}).
\end{equation}

Additionally, we obtain
\begin{equation}\label{gradiente Case 1}
\begin{array}{ll}
\displaystyle \int_{A_{\epsilon}} \epsilon^{-2} \varphi_{0, \tau}^{2} dV_{g} & \leq \displaystyle (1 + \epsilon C) \epsilon^{-2} \int_{A_{\epsilon}} \varphi_{0, \tau}^{2} dx  \vspace{0.2cm}\\ 
 & \leq (1+ \epsilon C) \epsilon^{2-n -2m} \tau^{\frac{n}{m+n} + m +\frac{n -4}{2} }
\end{array}
\end{equation} and
\begin{equation}\label{gradiente Case 2}
\begin{array}{ll}
\displaystyle \int_{A_{\epsilon}} |\nabla \varphi_{0, \tau}|^{2}dV_{g}& \leq \displaystyle (1 + \epsilon C) \int_{A_{\epsilon}} |\nabla \varphi_{0, \tau}|^{2} dx \vspace{0.2cm}\\ 
 &\leq  (1+ \epsilon C) \epsilon^{2-n -2m} \tau^{\frac{n}{m+n} + m +\frac{n -4}{2} }.
\end{array}
\end{equation}

Then,
\begin{equation}\label{gradiente Case 3}
\begin{array}{ll}
\displaystyle \int_{A_{\epsilon}}  |\nabla f_{\tau}|_{g}^{2}dxdt & = O(\epsilon^{2-n -2m} \tau^{\frac{n}{m+n} + m +\frac{n -4}{2} }).
\end{array}
\end{equation}

Next, we estimate the integral with the gradient term in $B_{\epsilon}$ in equality \eqref{w8}. Then, the symmetries of the ball, equality \eqref{g -1}, equality \eqref{dV normal 2}, Taylor expansion \eqref{Taylor v}, Taylor expansion \eqref{Taylor R}, and $\nabla_{g} v^{m}(p) = 0$ imply that \small
\begin{equation}\label{gradiente Case 4}
\begin{array}{ll}
\displaystyle \int_{B_{\epsilon}} |\nabla f_{\tau}|_{g}^{2} v^{m} dV_{g}   = & \displaystyle  \int_{B_{\epsilon}} |\nabla \varphi_{0, \tau}|^{2} dx + \frac{\Delta_g v^m}{2n} \int_{B_{\epsilon}}  |\nabla \varphi_{0, \tau}|^{2} |x|^2 dx \vspace{0.2cm} \\
& \displaystyle -  \frac{1}{3} R_{iklj} \int_{B_{\epsilon}}  (\partial_{i} \varphi_{0, \tau}) (\partial_{j} \varphi_{0, \tau})   x_{k} x_{l} dx \vspace{0.2cm} \\
& \displaystyle - \frac{1}{6} R_{iklj}(v^m)_{su} \int_{B_{\epsilon}}  (\partial_{i} \varphi_{0, \tau}) (\partial_{j} \varphi_{0, \tau})  x_{k} x_{l} x_{s} x_{u} dx \vspace{0.2cm} \\
& \displaystyle -  (  \frac{1}{20} R_{iklj,su} - \frac{3}{45} R_{iklr} R_{rsuj} )  \int_{B_{\epsilon}}  (\partial_{i} \varphi_{0, \tau}) (\partial_{j} \varphi_{0, \tau})  x_{k} x_{l} x_{s} x_{u} dx \vspace{0.2cm} \\
& \displaystyle + \frac{1}{24} (v^m)_{ijkl} \int_{B_{\epsilon}} |\nabla \varphi_{0, \tau}|^{2}x_{i} x_{j} x_{k} x_{l} dx \vspace{0.2cm} \\
& \displaystyle - (\frac{1}{40} R_{ij,kl}  + \frac{1}{180} R_{rijs} R_{rkls} ) \int_{B_{\epsilon}} |\nabla \varphi_{0, \tau}|^{2}x_{i} x_{j} x_{k} x_{l} dx \vspace{0.2cm} \\
& + \displaystyle \int_{B_{\epsilon}}  |\nabla \varphi_{0, \tau}|^{2} O(|x|^6) dx.\\
\end{array}
\end{equation} \normalsize

For the second integral in \eqref{gradiente Case 4}, we have
\begin{equation}\label{gradiente Case 5}
\begin{array}{ll}
\displaystyle \int_{B_{\epsilon}}  |\nabla \varphi_{0, \tau}|^{2} |x|^2  dx
& =\tau^{\frac{n}{m+n}} \frac{ (m+n-2)^2}{c(m,n)^{\frac{n}{2}}} \displaystyle \int_{B_{\frac{\epsilon\sqrt{c(m,n)}}{\sqrt{\tau}}}} |y|^4(1 + |y|^2)^{-(m+n)}dy\\
&=  \tau^{\frac{n}{m+n}} \frac{ (m+n-2)^2}{c(m,n)^{\frac{n}{2}}}  I_3 + O(\epsilon^{4-2m-n}\tau^{\frac{n}{m+n} + m + \frac{n-4}{2}}).
\end{array}
\end{equation}

Using the symmetries of the ball and of the Riemann curvature tensor, we obtain \small
\begin{equation}\label{gradiente Case 7}
\begin{array}{rl}
\dfrac{\int_{B_{\epsilon}}  (\partial_{i} \varphi_{0, \tau}) (\partial_{j} \varphi_{0, \tau})  R_{iklj}(0) x_{k} x_{l} dx}{(m+n-2)^2 c(m,n)^2\tau^{- \frac{n(m+n-2)}{2(m+n)} -2}}  
= &  \displaystyle \int_{B_{\epsilon}}  \frac{R_{iklj} x_{i} x_{j} x_{k} x_{l} }{(1 + \frac{c(m,n)}{\tau}|x|^2)^{m+n} } dx \\
 = & \displaystyle \int_{B_{\epsilon}} \dfrac{\sum\limits_{i=1}^{n}R_{iiii} x_i^4 + \sum\limits_{i \neq j}^{n}(R_{iijj} + R_{ijij} + R_{ijji} )  x_{i}^2 x_{j}^2 }{(1  + \frac{c(m,n)}{\tau} |x|^{2})^{m+n}} dx\\
= & 0.
\end{array}
\end{equation} \normalsize

Again, the symmetries of the ball and of the Riemann curvature tensor yield \small
\begin{equation}\label{gradiente Case 1 orden 2}
\begin{array}{l}
R_{iklj}(0)(v^m)_{su}(0)\displaystyle \int_{B_{\epsilon}}  (\partial_{i} \varphi_{0, \tau}) (\partial_{j} \varphi_{0, \tau})  x_{k} x_{l} x_{s} x_{u} dx = 0
\end{array}
\end{equation} \normalsize and \small
\begin{equation}\label{gradiente Case 2 orden 2}
\begin{array}{l}
(\frac{1}{20} R_{iklj,su}(0) - \frac{3}{45} R_{iklr}(0) R_{rsuj}(0) ) \displaystyle  \int_{B_{\epsilon}}  (\partial_{i} \varphi_{0, \tau}) (\partial_{j} \varphi_{0, \tau})  x_{k} x_{l} x_{s} x_{u} dx  = 0.
\end{array}
\end{equation}

\normalsize

In order to compute the sixth integral on the right-hand side of  \eqref{gradiente Case 4}, we obtain\small
\begin{equation}\label{gradiente Case 3 orden 2}
\begin{array}{rl}
\displaystyle  (v^m)_{ijkl}(0) \int_{B_{\epsilon}} |\nabla \varphi_{0, \tau}|^{2}x_{i} x_{j} x_{k} x_{l} dx \vspace{0.2cm} = & \displaystyle (v^m)_{ijkl} \int_{B_{\epsilon}} \frac{|x|^2 x_{i} x_{j} x_{k} x_{l}}{(1+ \frac{c(m,n)}{\tau}|x|^2)^{m+n}} dx. 
\end{array}
\end{equation} \normalsize 

The symmetries of the ball, Lemma \ref{integral con simetria} in the appendix, and $(v^m)_{ijji} = (v^m)_{jiji}$ imply \footnotesize
\begin{equation}
\begin{array}{rl}
\displaystyle \int_{B_{\epsilon}} \frac{(v^m)_{ijkl}(0) x_{i} x_{j} x_{k} x_{l}}{|x|^{-2}(1+ \frac{c(m,n)}{\tau}|x|^2)^{m+n}} dx= &\displaystyle \int_{B_{\epsilon}} \dfrac{\sum \limits_{i = 1}^{n} (v^m)_{iiii} x_i^{4} + \sum \limits_{i \neq j}^{n}  ((v^m)_{iijj} +  (v^m)_{ijij} + (v^m)_{ijji}) x_i^{2} x_j^{2}}{(1+ \frac{c(m,n)}{\tau}|x|^2)^{m+n}|x|^{-2} } dx \vspace{0.2cm}\\
= & \displaystyle \frac{1}{3} ((v^m)_{iijj} +   (v^m)_{ijij} + (v^m)_{ijji} ) \int_{B_{\epsilon}} \frac{ |x|^{2}x_i^{4} }{(1+ \frac{c(m,n)}{\tau}|x|^2)^{m+n} } dx \vspace{0.2cm}\\
=  & \displaystyle \frac{((v^m)_{iijj} +   (v^m)_{ijij} + (v^m)_{ijji} ) }{n(n+2)} \int_{B_{\epsilon}} \frac{ |x|^{6} }{(1+ \frac{c(m,n)}{\tau}|x|^2)^{m+n} } dx \vspace{0.2cm}\\
& \displaystyle \frac{((v^m)_{iijj} +   (v^m)_{ijij} + (v^m)_{jiji} ) }{n(n+2)} \int_{B_{\epsilon}} \frac{ |x|^{6} }{(1+ \frac{c(m,n)}{\tau}|x|^2)^{m+n} } dx. 
\end{array}
\end{equation} \normalsize

Formula \eqref{doble laplaciano vm} for $k= m$ implies
\begin{equation}\label{gradiente Case 6 orden 2}
\begin{array}{rl}
\displaystyle  (v^m)_{ijkl}(0) \int_{B_{\epsilon}} |\nabla \varphi_{0, \tau}|^{2}x_{i} x_{j} x_{k} x_{l} dx \vspace{0.2cm} = &  \displaystyle \frac{3\tau^{\frac{n}{m+n} + 1}(m+n-2)^2 \Delta^2 v^m  }{n(n+2)c(m,n)^{\frac{n+2}{2}} } I_4 \\ & + O(\epsilon^{6-2m-n} \tau^{\frac{n}{m+n} +m+ \frac{n-4}{2}}).
\end{array}
\end{equation}

Using a similar argument to obtain equality \eqref{gradiente Case 6 orden 2} and the contraction of Bianchi's identity $R_{g}, _{\, i} = 2 R^{j}_{\, \, i,j}$, we have
\begin{equation}\label{gradiente Case 7 orden 2}
\begin{array}{rl}
\displaystyle  R_{ij,kl} (0)  \int_{B_{\epsilon}} |\nabla \varphi_{0, \tau}|^{2}x_{i} x_{j} x_{k} x_{l} dx \vspace{0.2cm} 
= & \displaystyle \frac{\tau^{\frac{n}{m+n} + 1} (R_{ii,jj} +R_{ij,ij} +R_{ij,ji} )  }{n(n+2) c(m,n)^{\frac{n+2}{2}} (m+n-2)^{-2}} I_4 \vspace{0.2cm} \\
& + O(\epsilon^{6-2m-n} \tau^{\frac{n}{m+n} + m+ \frac{n-4}{2}}) \vspace{0.2cm}\\
= &\displaystyle \frac{2\tau^{\frac{n}{m+n} + 1}(m+n-2)^2 \Delta R_{g}  }{n(n+2) c(m,n)^{\frac{n+2}{2}} } I_4 \vspace{0.2cm}\\
& + O(\epsilon^{6-2m-n} \tau^{\frac{n}{m+n} + m+ \frac{n-4}{2}}).  \\
\end{array}
\end{equation}

The identities $R_{ijkl}R^{ijkl} = 2R_{ijkl}R^{ilkj}$ and $Ric_g(p)=0$ yield
\begin{equation}\label{gradiente Case 8 orden 2}
\begin{array}{rl}
\displaystyle R_{rijs}(0) R_{rkls}(0) \int_{B_{\epsilon}} |\nabla \varphi_{0, \tau}|^{2}x_{i} x_{j} x_{k} x_{l} dx 
=  & \displaystyle \frac{\tau^{\frac{n}{m+n} + 1} (R_{rs} R_{rs} + R_{irsj} R_{irsj} + R_{irsj} R_{isrj})}{n(n+2)c(m,n)^{\frac{n+2}{2}} (m+n-2)^{-2}} I_4  \vspace{0.2cm}  \\ 
& + O(\epsilon^{6-2m-n} \tau^{\frac{n}{m+n} + m+ \frac{n-4}{2}})  \vspace{0.2cm} \\
= & \displaystyle \frac{3\tau^{\frac{n}{m+n} + 1} (m+n-2)^{2} R_{irsj} R_{irsj}}{2n(n+2)c(m,n)^{\frac{n+2}{2}} } I_4  \vspace{0.2cm}\\
& + O(\epsilon^{6-2m-n} \tau^{\frac{n}{m+n} + m+ \frac{n-4}{2}}).
\end{array}
\end{equation}

For the last integral on the right-hand side of \eqref{gradiente Case 4}, 
taking $q < \min \{ 2m+ n-6, 1\}$ and $\epsilon < 1$ as in \eqref{R4 Case}, we obtain
\begin{equation}\label{gradiente Case 8}
\begin{array}{ll}
\displaystyle \int_{B_{\epsilon}}  |\nabla \varphi_{0, \tau}|^{2} |x|^6 dx & \leq \displaystyle \int_{B_{\epsilon}}  |\nabla \varphi_{0, \tau}|^{2} |x|^{4+ q} dx 
\leq  C\tau^{\frac{n}{m+n} + 1+ \frac{q}{2}}.
\end{array}
\end{equation}

Equalities \eqref{Laplacian vk} for $k=m$, \eqref{gradiente Case 3}, \eqref{gradiente Case 4}, \eqref{gradiente Case 5}, \eqref{gradiente Case 7}, \eqref{gradiente Case 1 orden 2}, \eqref{gradiente Case 2 orden 2},
\eqref{gradiente Case 6 orden 2}, \eqref{gradiente Case 7 orden 2}, and \eqref{gradiente Case 8 orden 2} and inequality \eqref{gradiente Case 8} lead to \eqref{gradiente Case 6} where
\begin{equation}\label{A1}
\begin{array}{l}
A_1 =  \displaystyle  - \frac{\Delta^2 v^m}{4} + \frac{\Delta R_{g}}{10}   + \frac{R_{iklj} R_{iklj}}{60}.
\end{array}
\end{equation}

By equality \eqref{doble laplaciano vm} for $k=m$, we obtain
\begin{equation}\label{A1.1}
\begin{array}{l}
A_1 =  \displaystyle   - \frac{m(m-1) |\text{Hess} \, v|_{g}^2}{2} - \frac{m(m-1) (\Delta_{g} v)^2}{4}   - \frac{m \Delta_{g}^2 v }{4}   + \frac{\Delta R_{g}}{10}   + \frac{R_{iklj} R_{iklj}}{60}.
\end{array}
\end{equation}

Finally, by hypothesis $Ric_g(p) = 0$, and formulas \eqref{T} and \eqref{W}, we conclude that $A_1$ takes the form of \eqref{A1 2}.
\end{proof}

Next, we analyze the last integral on the right-hand side of \eqref{w8} 
\begin{lema}\label{Estimate f vm-1}
Let $(M^n, g, v^{m} dV_g)$ be a compact smooth metric measure space such that for $p \in M$ we have $v(p) = 1$, $\nabla_{g} v(p) = 0$, $Ric_{g}(p) = 0$. If $f_{\tau}$ is defined by \eqref{f tau}, then
\begin{equation}\label{last integral Case 5}
\begin{array}{ll}
\displaystyle \int_{B_{2\epsilon}} f_{\tau}^{\frac{2(m+n-1)}{m+n-2}} v^{m-1} dV_{g}  = & \displaystyle \int_{B_{\epsilon}} \varphi_{0, \tau}^{\frac{2(m+n-1)}{m+n-2}} dx  +  
\dfrac{\tau^{\frac{n}{2(m+n)} + 1} (m-1)\Delta_{g}v}{2nc(m,n)^{\frac{n+2}{2}}} I_5 \vspace{0.2cm} \\
& - \dfrac{\tau^{\frac{n}{2(m+n)} + 2}  I_6 }{2n(n+2) c(m,n)^{\frac{n+4}{2}}} A_2 \vspace{0.2cm}\\
& +  O(\tau^{\frac{n}{2(m+n)} + m + \frac{n-2}{2}} \epsilon^{4-2m-n}) + O(\tau^{\frac{n}{2(m+n)} +2 + \frac{q}{2}})
\end{array}
\end{equation} where
\begin{equation}\label{I3 Case}
I_5 = \int_{\R^n} |y|^2(1  + |y|^2)^{-(m+n-1)} dy,
\end{equation}
\begin{equation}\label{I8 Case}
I_6 = \displaystyle \int_{\R^n} |y|^{4}(1 +  |y|^2)^{-(m+n-1)}  dy,
\end{equation} and \small
\begin{equation}\label{A2}
\begin{array}{ll}
A_2  = &\displaystyle  - \frac{(m-1)(m-2)}{2} |\text{Hess} \, v|_{g}^2 -  \frac{(m-1)(m-2)}{4} (\Delta_{g} v)^2 - \frac{(m-1)}{4} \Delta_{g}^2 v + \frac{\Delta R_{g}}{10} + \frac{|W|_g^2}{60}.
\end{array}
\end{equation} \normalsize
\end{lema}

\begin{proof} In the region $A_{\epsilon}$, we have
\begin{equation}\label{last integral Case 1}
\begin{array}{ll}
\displaystyle \int_{A_{\epsilon}} f_{\tau}^{\frac{2(m+n-1)}{m+n-2}} v^{m-1} dV_{g}  & \leq C(1+ \epsilon C) \displaystyle\int_{A_{\epsilon}} \varphi_{0, \tau}^{\frac{2(m+n-1)}{m+n-2}} dx \vspace{0.2cm}\\
 & \leq C(1+ \epsilon C) \tau^{\frac{n}{2(m+n)} + m + \frac{n-2}{2}} \epsilon^{2-2m-n}.\\
\end{array}
\end{equation}

In order to estimate the integral on the left-hand size of \eqref{last integral Case 5} in the region $B_{\epsilon}$, we use equality \eqref{dV normal}, Taylor expansion around $p$ for $v^{m-1}$, Taylor expansion \eqref{Taylor R}, $\nabla v(p) =0$, equality \eqref{Laplacian vk} for $k = m-1$, and the symmetries in the ball to obtain \small
\begin{equation}\label{last integral Case 2}
\begin{array}{rl}
\displaystyle \int_{B_{\epsilon}} f_{\tau}^{\frac{2(m+n-1)}{m+n-2}} v^{m-1} dV_{g} = & \displaystyle \int_{B_{\epsilon}} \varphi_{0, \tau}^{\frac{2(m+n-1)}{m+n-2}} dx + 
\frac{\Delta_{g} v^{m-1}}{2n} \int_{B_{\epsilon}}  \varphi_{0, \tau}^{\frac{2(m+n-1)}{m+n-2}} |x|^2 dx \vspace{0.2cm}\\ 
& \displaystyle + \frac{1}{24} (v^{m-1})_{ijkl} \int_{B_{\epsilon}}  \varphi_{0, \tau}^{\frac{2(m+n-1)}{m+n-2}} x_{i} x_{j} x_{k} x_{l} dx  \vspace{0.2cm}\\
&\displaystyle - (\frac{1}{40} R_{ij,kl} + \frac{1}{180} R_{rijs} R_{rkls} ) \int_{B_{\epsilon}}  \varphi_{0, \tau}^{\frac{2(m+n-1)}{m+n-2}} x_{i} x_{j} x_{k} x_{l} dx  \vspace{0.2cm}\\
& \displaystyle + \int_{B_{\epsilon}}  \varphi_{0, \tau}^{\frac{2(m+n-1)}{m+n-2}} O(|x|^6) dx.
\end{array}
\end{equation} \normalsize

Now, the second term on the right-hand side of \eqref{last integral Case 2} takes the form
\begin{equation}\label{last integral Case 3}
\begin{array}{ll}
\displaystyle \int_{B_{\epsilon}}  \varphi_{0, \tau}^{\frac{2(m+n-1)}{m+n-2}} |x|^2 dx  
& = \dfrac{\tau^{\frac{n}{2(m+n)} + 1}}{c(m,n)^{\frac{n+2}{2}}} I_5 + O(\tau^{\frac{n}{2(m+n)} + m + \frac{n-2}{2}} \epsilon^{4-2m-n}).\\
\end{array}
\end{equation}

For the third term on the right-hand side of \eqref{last integral Case 2}, a similar argument as in \eqref{gradiente Case 3 orden 2} to \eqref{gradiente Case 6 orden 2} using formula \eqref{doble laplaciano vm} for $k = m-1$ yields
\begin{equation}\label{last integral Case 5 order 2}
\begin{array}{rl}
(v^{m-1})_{ijkl}(0) \displaystyle \int_{B_{\epsilon}}  \varphi_{0, \tau}^{\frac{2(m+n-1)}{m+n-2}} x_{i} x_{j} x_{k} x_{l} dx = & \dfrac{3\tau^{\frac{n}{2(m+n)} + 2}  \Delta^2 v^{m-1} }{n(n+2) c(m,n)^{\frac{n+4}{2}}} I_6 \vspace{0,2cm}\\
& + O(\epsilon^{6-2m-n} \tau^{\frac{n}{2(m+n)} + m + \frac{n-2}{2}}). 
\end{array}
\end{equation}

Using the contraction of Bianchi's identity $R_{g},i = 2 R^{j}_{\, \, i,j}$, we obtain
\begin{equation}\label{last integral Case 6 order 2}
\begin{array}{rl}
R_{ij,kl} (0)\displaystyle \int_{B_{\epsilon}}  \varphi_{0, \tau}^{\frac{2(m+n-1)}{m+n-2}} x_{i} x_{j} x_{k} x_{l} dx 
= & \dfrac{2\tau^{\frac{n}{2(m+n)} + 2} \Delta R_{g} }{n(n+2) c(m,n)^{\frac{n+4}{2}}} I_6 \vspace{0,2cm} \\
& + O(\epsilon^{6-2m-n} \tau^{\frac{n}{2(m+n)} + m + \frac{n-2}{2}}), \\
\end{array}
\end{equation}

The identity $R_{ijkl}R^{ijkl} = \frac{1}{2}R_{ijkl}R^{ilkj}$ and $Ric_g(p)=0$ imply
\small
\begin{equation}\label{last integral Case 7 order 2}
\begin{array}{ll}
R_{rijs} (0) R_{rkls} (0) \displaystyle \int_{B_{\epsilon}}  \varphi_{0, \tau}^{\frac{2(m+n-1)}{m+n-2}} x_{i} x_{j} x_{k} x_{l} dx  = & \dfrac{3\tau^{\frac{n}{2(m+n)} + 2}R_{rijs}  R_{rijs}}{2n(n+2) c(m,n)^{\frac{n+4}{2}}} I_6 \vspace{0,2cm}\\
&  +  O(\epsilon^{6-2m-n} \tau^{\frac{n}{2(m+n)} + m + \frac{n-2}{2}}). \\
\end{array}
\end{equation} \normalsize

On the other hand, since $\epsilon <1 $ and we choose $0 < q < \min\{2m +n -6, 1\}$, 
the last term on the right-hand side of \eqref{last integral Case 2} is estimated as follows
\begin{equation}\label{last integral Case 4}
\begin{array}{ll}
\displaystyle \int_{B_{\epsilon}}  \varphi_{0, \tau}^{\frac{2(m+n-1)}{m+n-2}} |x|^6 dx  & \leq \tau^{-\frac{n(m+n-1)}{2(m+n)}} \displaystyle \int_{B_{\epsilon}} |x|^{4 + q}(1  + \frac{c(m,n)}{\tau}|x|^2)^{-(m+n-1)} dx \\ 
& \leq C \tau^{\frac{n}{2(m+n)} + 2 +\frac{q}{2}}.
\end{array}
\end{equation}

Equality \eqref{Laplacian vk} and estimates \eqref{last integral Case 1} to \eqref{last integral Case 4} allow us to conclude formula \eqref{last integral Case 5}, where
\begin{equation}\label{A2*}
\begin{array}{ll}
A_2 = &\displaystyle  - \frac{\Delta_g^2 v^{m-1}}{4} + \frac{\Delta_g R_{g}}{10}  + \frac{R_{iklj} R_{iklj}}{60}.
\end{array}
\end{equation}

By hypothesis $Ric_g(p)=0$, equality \eqref{W}, and equality \eqref{doble laplaciano vm} for $k=m-1$, it follows that $A_2$ satisfies equality \eqref{A2}.
\end{proof}

Now, we analyze the behavior of $V_{\tau}$, $V_{\tau}^{-\frac{m+n-2}{m+n}}$, and $V_{\tau}^{-\frac{m+n-1}{m+n}}$ when $\tau$ is near zero.

\begin{lema}\label{Estimate V}
Let $(M^n, g, v^{m} dV_g)$ be a compact smooth metric measure space such that for $p \in M$ we have $v(p) = 1$, $\nabla_{g} v(p) = 0$, and $Ric_{g}(p) = 0$. If $V$, $f_{\tau}$, and  $V_{\tau}$ are defined by \eqref{Definition V}, \eqref{f tau}, and \eqref{definition Vtau}, respectively; then, 
\begin{equation}\label{V - Vt tilde 7}
\begin{array}{ll}
V - V_{\tau} = &  - \dfrac{\tau m\Delta_{g}v}{2n c(m,n)^{\frac{n+2}{2}}}  I_7 + \dfrac{\tau^{2}  I_3 A_1}{2n(n+2) c(m,n)^{\frac{n+4}{2}}}+  O(\tau^{ m + \frac{n}{2}} \epsilon^{-2m-n}) + O(\tau^{2 + \frac{q}{2}}),
\end{array}
\end{equation} where $A_1$ is defined by \eqref{A1 2} and
\begin{equation}\label{I4 Case}
I_7 = \int_{\R^n} |y|^2 (1 + |y|^2)^{-(m+n)} dy.
\end{equation}

Additionally, we obtain the following estimates
\begin{equation}\label{V - Vt potencia tilde 1}
\begin{array}{ll}
V_{\tau}^{-\frac{m+n-2}{m+n}}   = & V^{-\frac{m+n-2}{m+n}}  - \dfrac{\tau m(m+n-2) I_7 \Delta_g v }{2n(m+n)c(m,n)^{\frac{n+2}{2}}V^{\frac{2m+2n-2}{m+n}}} \vspace{0.2cm}\\
& + \dfrac{\tau^{2}  (m+n-2) I_3 A_1}{2n(n+2)(m+n) c(m,n)^{\frac{n+4}{2}} V^{\frac{2m+2n-2}{m+n}}} \vspace{0.2cm} \\
& + \dfrac{\tau^{2}  m^2(m+n-2)(m+n-1)   (I_7)^2 (\Delta_{g}v)^2  }{4 n^2 (m+n)^2 c(m,n)^{n+2} V^{\frac{3m+3n-2}{m+n}}} \vspace{0.2cm} \\
& +  O(\tau^{ m + \frac{n}{2}} \epsilon^{-2m-n}) + O(\tau^{2 + \frac{q}{2}})
\end{array}
\end{equation} and 
\begin{equation}\label{V - Vt potencia tilde 2}
\begin{array}{ll}
V_{\tau}^{-\frac{m+n-1}{m+n}}  & = V^{-\frac{m+n-1}{m+n}}   - \dfrac{\tau m (m+n-1)  \Delta_g v }{2n(m+n)c(m,n)^{\frac{n+2}{2}}V^{\frac{2m+2n-1}{m+n}}} I_7 \vspace{0.2cm} \\
& + \dfrac{\tau^{2}  (m+n-1) I_3 A_1}{2n(n+2)(m+n) c(m,n)^{\frac{n+4}{2}} V^{\frac{2m+2n-1}{m+n}}} \vspace{0.2cm}\\
& + \dfrac{\tau^{2}  m^2(m+n-1)(2m+2n-1) (\Delta_{g}v)^2 (I_7)^2 }{8 n^2 (m+n)^2 c(m,n)^{n+2} V^{\frac{3m+3n-1}{m+n}}} \vspace{0.2cm}\\
& +  O(\tau^{ m + \frac{n}{2}} \epsilon^{-2m-n}) + O(\tau^{2 + \frac{q}{2}}).
\end{array}
\end{equation}
\end{lema}

\begin{proof} By the definitions of $V$ and $V_{\tau}$, we obtain
\begin{equation}\label{V - Vt tilde 1}
\begin{array}{ll}
V - V_{\tau}  = & \displaystyle \int_{\R^n \setminus B_{2\epsilon}} \varphi_{0, \tau}^{\frac{2(m+n)}{m+n-2}}dx + (\int_{A_{\epsilon}} \varphi_{0, \tau}^{\frac{2(m+n)}{m+n-2}} dx - \int_{ A_{\epsilon}} f_{0, \tau}^{\frac{2(m+n)}{m+n-2}}v^{m}dV_{g}) \vspace{0.2cm}\\
& \displaystyle + (\int_{B_{\epsilon}} \varphi_{0, \tau}^{\frac{2(m+n)}{m+n-2}} dx - \int_{B_{\epsilon}} \varphi_{0, \tau}^{\frac{2(m+n)}{m+n-2}}v^{m}dV_{g}).
\end{array}
\end{equation} 

For the first integral on the right-hand side of \eqref{V - Vt tilde 1} we have
\begin{equation}\label{V - Vt tilde 2}
\begin{array}{ll}
\displaystyle \int_{\R^n \setminus B_{2\epsilon}} \varphi_{0, \tau}^{\frac{2(m+n)}{m+n-2}}dx 
& \leq C \epsilon^{-n-2m}\tau^{m + \frac{n}{2} }.
\end{array}
\end{equation}

Using the expansion for the volume form \eqref{dV normal}, and the fact that $v$ is bounded, we have in the second integral on the right-hand side of \eqref{V - Vt tilde 1} that \begin{equation}\label{V - Vt tilde 3}
\begin{array}{ll}
\left |\displaystyle\int_{A_{\epsilon}} \varphi_{0, \tau}^{\frac{2(m+n)}{m+n-2}} dx - \int_{ A_{\epsilon}} f_{0, \tau}^{\frac{2(m+n)}{m+n-2}}v^{m}dV_{g} \right| & \leq C(1+ \epsilon C) \displaystyle \int_{A_{\epsilon}} \varphi_{0, \tau}^{\frac{2(m+n)}{m+n-2}} dx  \vspace{0.2cm}\\
& \leq C(1+ \epsilon C) \epsilon^{-n-2m}\tau^{m + \frac{n}{2} }.
\end{array}
\end{equation}

By the expansion for the volume form \eqref{dV normal}, the Taylor expansion around $p$ for $v^{m}$ and the symmetries of the ball in the third integral on the right-hand side of \eqref{V - Vt tilde 1}, we obtain \small
\begin{equation}\label{V - Vt tilde 4}
\begin{array}{rl}
\displaystyle \int_{B_{\epsilon}} \varphi_{0, \tau}^{\frac{2(m+n)}{m+n-2}} dx - \int_{B_{\epsilon}} \varphi_{0, \tau}^{\frac{2(m+n)}{m+n-2}}v^{m}dV_{g}   = & 
-\dfrac{\Delta_{g}v^m}{2n} \displaystyle \int_{B_{\epsilon}} \varphi_{0, \tau}^{\frac{2(m+n)}{m+n-2}} |x|^2 dx \vspace{0.2cm}\\
 & - \displaystyle \frac{1}{24} (v^{m})_{ijkl}\int_{B_{\epsilon}}  \varphi_{0, \tau}^{\frac{2(m+n)}{m+n-2}} x_{i} x_{j} x_{k} x_{l} dx  \vspace{0.2cm}\\
 & + \displaystyle (\frac{1}{40} R_{ij,kl} + \frac{1}{180} R_{rijs} R_{rkls}) \int_{B_{\epsilon}}  \varphi_{0, \tau}^{\frac{2(m+n)}{m+n-2}} x_{i} x_{j} x_{k} x_{l} dx  \vspace{0.2cm}\\
& + \displaystyle \int_{B_{\epsilon}} \varphi_{0, \tau}^{\frac{2(m+n)}{m+n-2}} O(|x|^6) dx.\\
\end{array}
\end{equation} \normalsize

To analyze \eqref{V - Vt tilde 4}, we consider the first integral on its right-hand side
\begin{equation}\label{V - Vt tilde 5}
\begin{array}{ll}
\displaystyle\int_{B_{\epsilon}} \varphi_{0, \tau}^{\frac{2(m+n)}{m+n-2}} |x|^2 dx  
& = \displaystyle \frac{\tau}{c(m,n)^{\frac{n+2}{2}}} I_7 + O(\epsilon^{2-2m-n} \tau^{m + \frac{n}{2}}).
\end{array}
\end{equation} 

Additionally, we have
\begin{equation}\label{V - Vt tilde 5*}
\begin{array}{ll}
\displaystyle\int_{B_{\epsilon}} \varphi_{0, \tau}^{\frac{2(m+n)}{m+n-2}} |x|^4 dx  
& = \displaystyle \frac{\tau^2}{c(m,n)^{\frac{n+4}{2}}} I_3 + O(\epsilon^{4-2m-n} \tau^{m + \frac{n-4}{2}}).
\end{array}
\end{equation}

For the last integral on the right-hand side of \eqref{V - Vt tilde 4}, recalling that $q < \min \{2m-n-4, 1 \}$ and $\epsilon < 1$, we obtain
\begin{equation}\label{V - Vt tilde 6}
\begin{array}{ll}
\displaystyle \int_{B_{\epsilon}} \varphi_{0, \tau}^{\frac{2(m+n)}{m+n-2}} |x|^6 dx 
&\leq C \tau^{2 + \frac{q}{2}}.
\end{array}
\end{equation}

Equalities \eqref{Laplacian vk} for $k=m$, \eqref{V - Vt tilde 1}, \eqref{V - Vt tilde 4}, \eqref{V - Vt tilde 5}, and \eqref{V - Vt tilde 5*}; inequalities \eqref{V - Vt tilde 2}, \eqref{V - Vt tilde 3}, and  \eqref{V - Vt tilde 6}; and similar arguments like we used in \eqref{last integral Case 5 order 2} to \eqref{last integral Case 7 order 2} lead to formula \eqref{V - Vt tilde 7}.

It follows that the terms $V_{\tau}$ are uniformly bounded away from zero. Using estimate \eqref{V - Vt tilde 7} and Taylor expansion around $V$ for the functions $x^{-\frac{m+n-2}{m+n}}$ and $x^{-\frac{m+n-1}{m+n}}$, we obtain formulas \eqref{V - Vt potencia tilde 1} and \eqref{V - Vt potencia tilde 2}, respectively.
\end{proof}

\subsection{Proof of Theorem A} In this subsection, we prove Theorem A.

\textbf{\textit{Proof of Theorem A}}.
Let $(M^n, g, v^{m} dV_g)$ be a compact smooth metric measure space, where $(M,g)$ is nonlocally conformally flat. Then, there exist $p \in M$ such that the Weyl tensor in $p$ is nonzero, i.e. $|W|_g(p) \neq 0 $. Since $|W|_g$ is conformally invariant, and with the equality \eqref{W conformal Case} and Lemma \ref{cambio conforme especial}, we can assume without loss of generality that $v(p) = 1$, $\nabla_{g} v(p) = 0$, $\Delta_{g} v(p) = 0$, and $Ric_g(p) = 0$. Let $\tau >0$ and consider $f_{\tau}$, $\tilde{f}_{\tau}$, and  $V_{\tau}$, defined by \eqref{f tau}, \eqref{f tau tilde}, and \eqref{definition Vtau}, respectively.

Using the lemmas in subsection \ref{local expansion}, we estimate  $\mathcal{W}[M^n,  g, v^{m} dV_{g}](\tilde{f}_{\tau}, \tilde{\tau})$ where $\tilde{\tau} = \tau V^{-\frac{2}{2m+n}}$. First, we obtain
\begin{equation}\label{V - Vt potencia tilde 3}
\begin{array}{ll}
\dfrac{1}{(m+n-2)^2 } \displaystyle \int_{B_{\epsilon}} |\nabla \varphi_{0, \tau}|^{2}dx & = \frac{\tau^{\frac{n}{m+n} -1}}{ c(m,n)^{\frac{n-2}{2}}} I_7 + O(\epsilon^{2-2m-n} \tau^{\frac{n}{m+n} + m + \frac{n}{2}}).\\
\end{array}
\end{equation}

Estimate \eqref{gradiente Case 6} in Lemma \ref{Estimate gradiente}, estimate \eqref{V - Vt potencia tilde 1} in Lemma \ref{Estimate V},  $\tilde{\tau} = \tau V^{\frac{-2}{2m+n}}$, and equality \eqref{V - Vt potencia tilde 3} implies that
\begin{equation}\label{V - Vt potencia tilde 4}
\begin{array}{rl}
\displaystyle  \dfrac{\tilde{\tau}^{\frac{m}{m+n}}}{ V_{\tau}^{\frac{m+n-2}{m+n}}}  \int_{B_{2\epsilon}} |\nabla_g f_{\tau}|_g^{2}v^m dV_{g}
= & \displaystyle \dfrac{\tilde{\tau}^{\frac{m}{m+n}}}{V^{\frac{m+n-2}{m+n}}} \int_{B_{\epsilon}} |\nabla \varphi_{0, \tau}|^{2}dx \vspace{0.2cm}  \\
 & - \dfrac{\tau m (m+n-2)^3(I_7)^{2} \Delta_{g}v}{2n(m+n)c(m,n)^{n}V^{\frac{2m}{(m+n)(2m+n)} + \frac{2m+2n-2}{m+n}}} \vspace{0.2cm} \\
 &+ \dfrac{\tau^2 (m+n-2)^3 I_7 I_3 A_1}{2n(n+2)(m+n)c(m,n)^{n+1}V^{\frac{2m}{(m+n)(2m+n)} + \frac{2m+2n-2}{m+n}}} \vspace{0.2cm}\\
 &+ \dfrac{\tau^{2}  (m+n-2)^3(m+n-1) m^2 (I_7)^3 (\Delta_{g}v)^2}{4n^2(m+n)^2 c(m,n)^{\frac{3n+2}{2}} V^{\frac{2m}{(m+n)(2m+n)} + \frac{3m+3n-1}{m+n}}} \vspace{0.2cm}\\
 & + \dfrac{\tau m(m+n-2)^2  I_3 \Delta_{g}v }{2nc(m,n)^{\frac{n}{2}}V^{\frac{2m}{(m+n)(2m+n)} + \frac{m+n-2}{m+n}}} \vspace{0.2cm}\\
 &- \dfrac{\tau^2 m^2(m+n-2)^3  I_3 I_7 (\Delta_{g}v)^2 }{4n^2(m+n)c(m,n)^{n+1} V^{\frac{2m}{(m+n)(2m+n)} + \frac{2m+2n-2}{m+n}}} \vspace{0.2cm}\\
 &-\dfrac{\tau^2 (m+n-2)^2 I_4 A_1}{2n(n+2) c(m,n)^{\frac{n+2}{2}} V^{\frac{2m}{(m+n)(2m+n)} + \frac{m+n-2}{m+n}}} \vspace{0.2cm}\\
 &+  O(\tau^{ m + \frac{n}{2}} \epsilon^{-2m-n})  + O(\tau^{2 + \frac{q}{2}}).
\end{array}
\end{equation}

Equality $\tilde{\tau} = \tau V^{\frac{-2}{2m+n}}$, estimate \eqref{R5 Case} in Lemma \ref{Estimate R}, and estimate \eqref{V - Vt potencia tilde 1} in Lemma \ref{Estimate V} yield 
\begin{equation}\label{V - Vt potencia tilde 1 order 2}
\begin{array}{ll}
\displaystyle \dfrac{\tilde{\tau}^{\frac{m}{m+n}}}{ V_{\tau}^{\frac{m+n-2}{m+n}}} \int_{B_{2\epsilon} } R^{m}_{\phi} f_{\tau}^2 v^m dV_g
 = & -\dfrac{2\tau m\Delta_{g}v I_1 }{c(m,n)^{\frac{n}{2}} V^{\frac{2m}{(m+n)(2m+n)} + \frac{m+n-2}{m+n}}}  \vspace{0.2cm}\\
& + \dfrac{\tau^{2} m^2(m+n-2) (\Delta_{g}v)^2 I_1 I_7}{n (m+n) c(m,n)^{n+1} V^{\frac{2m}{(m+n)(2m+n)} + \frac{2m+2n-2}{m+n}}} \vspace{0.2cm}\\
&+  \dfrac{\tau^{2} ( \Delta_g  R^{m}_{\phi}  -2m^2  (\Delta_g v)^2 ) I_2 }{2nc(m,n)^{\frac{n+2}{2}} V^{\frac{2m}{(m+n)(2m+n)} + \frac{m+n-2}{m+n}}}  \vspace{0.2cm}\\
& +  O(\tau^{ m + \frac{n}{2}} \epsilon^{-2m-n}) + O(\tau^{2 + \frac{q}{2}}).
\end{array}
\end{equation} 

Now, we obtain
\begin{equation}\label{V - Vt potencia tilde 5}
\begin{array}{ll}
\displaystyle \int_{B_{\epsilon}} \varphi_{0, \tau}^{\frac{2(m+n-1)}{m+n-2}}dx
& = \frac{\tau^{\frac{n}{2(m+n)}}}{ c(m,n)^{\frac{n}{2}}} I_8 + O(\epsilon^{2-2m-n} \tau^{\frac{n}{2(m+n)} + m + \frac{n-2}{2}})\\
\end{array}
\end{equation} where
\begin{equation}\label{I5 Case}
I_8 = \int_{\R^n} (1 + |y|^2)^{-(m+n-1)}dy.
\end{equation}

Equality $\tilde{\tau} = \tau V^{\frac{-2}{2m+n}}$, estimative \eqref{V - Vt potencia tilde 2} in Lemma \ref{Estimate V}, and equality \eqref{V - Vt potencia tilde 5} imply that \small
\begin{equation}\label{V - Vt potencia tilde 6}
\begin{array}{ll}
\displaystyle \dfrac{m \displaystyle \int_{B_{2\epsilon}} f_{\tau}^{\frac{2(m+n-1)}{m+n-2}}v^{m-1} dV_{g} }{\tilde{\tau}^{\frac{n}{2(m+n)}} V_{\tau}^{\frac{m+n-1}{m+n}}}  
 = & \displaystyle \dfrac{m \displaystyle \int_{B_{\epsilon}} \varphi_{0, \tau}^{\frac{2(m+n-1)}{m+n-2}}dx}{\tilde{\tau}^{\frac{n}{2(m+n)}} V^{\frac{m+n-1}{m+n}}}   \vspace{0.2cm}\\
& - \displaystyle \dfrac{\tau m^2(m+n-1) I_7 I_8 \Delta_{g}v}{2n(m+n)c(m,n)^{n+1} V^{\frac{-n}{(m+n)(2m+n)} + \frac{2m+2n-1}{m+n}}} \vspace{0.2cm}\\
& + \dfrac{\tau^2 m(m+n-1) I_8 I_3 A_1}{2n(n+2)(m+n)c(m,n)^{n+2} V^{\frac{-n}{(m+n)(2m+n)} + \frac{2m+2n-1}{m+n}}}  \vspace{0.2cm}\\
& + \dfrac{\tau^{2}  m^3(m+n-1)(2m+2n-1)  (I_7)^2 I_8(\Delta_{g}v)^2}{8n^2(m+n)^2 c(m,n)^{\frac{3n+4}{2}} V^{\frac{-n}{(m+n)(2m+n)} + \frac{3m+3n-1}{m+n}} } \vspace{0.2cm}\\
& + \dfrac{\tau m(m-1)I_5 \Delta_{g}v }{2n c(m,n)^{\frac{n+2}{2}} V^{\frac{-n}{(m+n)(2m+n)} + \frac{m+n-1}{m+n}}} \vspace{0.2cm}\\
& - \dfrac{\tau^{2}  m^2(m-1)(m+n-1)I_5 I_7 (\Delta_{g}v)^2}{4n^2(m+n) c(m,n)^{n+2} V^{\frac{-n}{(m+n)(2m+n)} + \frac{2m+2n-1}{m+n}} } \vspace{0.2cm}\\
& -\dfrac{\tau^2 m I_6 A_2}{2n(n+2) c(m,n)^{\frac{n+4}{2}}V^{\frac{-n}{(m+n)(2m+n)} + \frac{m+n-1}{m+n}} } \vspace{0.2cm} \\ 
& +  O(\tau^{ m + \frac{n-2}{2}} \epsilon^{-2m-n})  + O(\tau^{2 + \frac{q}{2}}).
\end{array}
\end{equation} \normalsize

The equality \eqref{nu Rn case} and the estimates \eqref{V - Vt potencia tilde 4}, \eqref{V - Vt potencia tilde 1 order 2}, and \eqref{V - Vt potencia tilde 6} imply that \eqref{w8} takes the following form 
\begin{equation}\label{w10}
\begin{array}{rl}
\mathcal{W}[M^n,  g, v^{m} dV_{g}](\tilde{f}_{\tau}, \tilde{\tau}) + m = & \nu[\R^n,  dx^2, dV, m] + m + \tau  A_3  + \tau^2 A_4 \vspace{0.2cm}\\
&+  O(\tau^{ m + \frac{n-2}{2}} \epsilon^{-2m-n}) + O(\tau^{2 + \frac{q}{2}})
\end{array}
\end{equation} where
\begin{equation}\label{A3}
\begin{array}{lll}
A_3  : = &  \displaystyle \frac{mV^{-\frac{2m}{(m+n)(2m+n)} - \frac{m+n-2}{m+n}} \Delta_{g}v}{2c(m,n)^{\frac{n}{2}}} \Bigg ( & \displaystyle \frac{(m+n-2)^2}{n}    I_3  + \frac{m-1}{n c(m,n)} I_5  - \frac{m+n-2}{(m+n-1)} I_1 \vspace{0.2cm}\\
& &\left. \displaystyle  - \frac{m(m+n-1) I_7 I_8 }{n(m+n)c(m,n)^{\frac{n+2}{2}} V} - \frac{(m+n-2)^3 (I_7)^{2}}{n(m+n)c(m,n)^{\frac{n}{2}}V} \right) 
\end{array}
\end{equation} and \small
\begin{equation}\label{A4}
\begin{array}{lll}
A_4 : = & \dfrac{V^{\frac{-n}{(m+n)(2m+n)} + \frac{m+n-1}{m+n}}}{c(m,n)^{\frac{n+2}{2}}} \bigg( & \dfrac{(m+n-2)^3 I_7 I_3 A_1}{2n(n+2)(m+n)c(m,n)^{\frac{n}{2}}V}  \vspace{0.2cm}\\
& &+ \dfrac{ m^2(m+n-2)^3(m+n-1) (\Delta_{g}v)^2 (I_7)^3 }{4 n^2 (m+n)^2 c(m,n)^{n} V^{2}}  \vspace{0.2cm}\\ 
& & - \dfrac{m^2(m+n-2)^3 (\Delta_{g}v)^2 I_3 I_7  }{4n^2(m+n)c(m,n)^{\frac{n}{2}}V}  -\dfrac{(m+n-2)^2 I_4 A_1}{2n(n+2) } 
\vspace{0.2cm}\\ 
& & + \dfrac{m^2(m+n-2)^2 (\Delta_{g}v)^2 I_1 I_7}{4 n (m+n)(m+n-1) c(m,n)^{\frac{n}{2}} V} \vspace{0.2cm}\\
& &+  \dfrac{ (m+n-2) ( \Delta_g  R^{m}_{\phi}  -2m^2  (\Delta_g v)^2 )I_2}{8n (m+n-1)} 
\vspace{0.2cm}\\ 
& &+ \dfrac{ m(m+n-1) I_8 I_3 A_1}{2n(n+2)(m+n)c(m,n)^{\frac{n+2}{2}} V} 
\vspace{0.2cm}\\
& &+ \dfrac{  m^3(m+n-1)(2m+2n-1)   (I_7)^2 I_8 (\Delta_{g} v)^2}{8 n^2 (m+n)^2 c(m,n)^{n+1} V^{2} } \vspace{0.2cm}\\
& &- \dfrac{  m^2(m-1)(m+n-1) I_5 I_7(\Delta_{g} v)^2}{4n^2(m+n) c(m,n)^{\frac{n+2}{2}} V}  \left. -\dfrac{ m I_6 A_2}{2n(n+2) c(m,n)  } \right).
\end{array}
\end{equation} \normalsize

Using the comparisons for integrals given in Lemma \ref{comparacion integrales} (see appendix below) and the equality $c(m,n) = \frac{m+n-1}{(m+n-2)^2}$, it follows that
\begin{equation}\label{A3 comparcion}
\begin{array}{rll}
A_3  = & \displaystyle \frac{ m(m+n-2)^2 \Delta_{g}v I_7}{2c(m,n)^{\frac{n}{2}} V^{\frac{2m}{(m+n)(2m+n)} + \frac{m+n-2}{m+n}}} \bigg (  & \displaystyle  \frac{n+2}{n(2m+n-4)} -\frac{4}{n(2m+n-4)}     \vspace{0.2cm}\\
& &\displaystyle  + \frac{2(m-1)}{n (2m+n-4)} - \frac{m+n-2}{ (m+n)(2m+n-2)} \vspace{0.2cm} \\
& &\displaystyle \left.   - \frac{2m(m+n-1)  }{n(m+n)(2m+n-2)}   \right) \\
= & 0. &
\end{array}
\end{equation}

Next, we analyze $A_4$. For this purpose, note that
\begin{equation}\label{computation m y n}
\begin{array}{ll}
\frac{m-5}{n(2m+n-4)(2m+n-6)} = & \frac{(m+n-2)}{2(m+n)(2m+n-2)(2m+n-4)} \vspace{0.2cm}\\
 &- \frac{ (n+4)}{2n(2m+n-4)(2m+n-6)} + \frac{m(m+n-1)}{n(m+n)(2m+n-2)(2m+n-4)}.
\end{array}
\end{equation} By equality \eqref{computation m y n} and Lemma \ref{comparacion integrales}, we obtain
\small

\begin{equation}\label{A4 comparcion}
\begin{array}{lll}
A_4 = & \displaystyle \dfrac{I_7 (m+n-2)^2 }{c(m,n)^{\frac{n+2}{2}}V^{\frac{n}{(m+n)(2m+n)} - \frac{m+n-1}{m+n}}} \bigg (& \dfrac{(m-5) A_1}{n(2m+n-4)(2m+n-6)}  \vspace{0.2cm}\\
& & + \dfrac{ m^2(m+n-2)(m+n-1) (\Delta_{g}v)^2 }{4(m+n)^2 (2m+n-2)^2} \vspace{0.2cm}\\
& &- \dfrac{m^2(m+n-2)(n+2) (\Delta_{g}v)^2  }{4 n(m+n)(2m+n-2)(2m+n-4)}  \vspace{0.2cm}\\
& & + \dfrac{m^2(m+n-2)(\Delta_{g}v)^2}{n (m+n) (2m+n-2) (2m+n-4)} \vspace{0.2cm}\\
& & +  \dfrac{ \Delta_g  R^{m}_{\phi}  -2m^2  (\Delta_g v)^2 }{2n (2m+n-4)(2m+n-6)}  \vspace{0.2cm}\\
& &     + \dfrac{ m^3(m+n-1)(2m+2n-1) (\Delta_{g}v)^2}{4 n(m+n)^2 (2m+n-2)^2 } \vspace{0.2cm}\\
& & - \dfrac{  m^2(m-1)(m+n-1)  (\Delta_{g}v)^2}{2n(m+n)(2m+n-2)(2m+n-4) }  \vspace{0.2cm}\\
& & \left. -\dfrac{m A_2}{n(2m+n-4) (2m+n-6)} \right).
\end{array}
\end{equation} \normalsize

The definitions of $A_1$ and $A_2$ in \eqref{A1 2} and \eqref{A2}, respectively, and equality \eqref{Delta Rm} imply that $A_4$ takes the form
\begin{equation}\label{A4 nuevo}
\begin{array}{l}
A_4 = \dfrac{\left( -\frac{1}{6}  |W|_g^2 + m(m-1)|\text{Hess} \, v|_{g}^2 \right)}{2n(2m+n-4)(2m+n-6)} + A_5(\Delta_{g}v)^2 
\end{array}
\end{equation} where 
\begin{equation}\label{A5}
\begin{array}{ll}
A_5  = & \frac{m(1-m)}{4n(2m+n-4)(2m+n-6)} + \frac{m^2(m+n-1)(m+n-2)}{4(m+n)^2(2m+n-2)^2} -  \frac{m^2(m+n-2)(n+2)}{4n(m+n)(2m+n-2)(2m+n-4)} \vspace{0.2cm}\\
&+\frac{m^2(m+n-2)}{n(m+n)(2m+n-2)(2m+n-4)}  +  \frac{m^3(m+n-1)(2m+2n-1)}{4n(m+n)^2(2m+n-2)^2} - \frac{m^2(m-1)(m+n-1)}{2n(m+n)(2m+n-2)(2m+n-4)}.
\end{array}
\end{equation} 

Since the problem coincides with the Yamabe problem when $m=0$, we analyze the term on the right-hand side of \eqref{A4 nuevo} for $m > 0$. Note that $m(m-1)|\text{Hess} \, v|_{g}^2$ is nonpositive if $0<m \leq 1$. For $n \geq 7$, we have $\lim \limits_{m \to 0} A_5 = 0$, and for $n=6$, the term $A_5$ is bounded and 
\[
\lim \limits_{m \to 0} \frac{-\frac{1}{6}  |W|_g^2(p) }{2n(2m+n-4)(2m+n-6)} = -\infty.
\]

Then, for $n\geq 6$ there exists $ 0 < \delta \leq 1$ such that $A_4<0$ for $m < \delta$.

 For $0 < m \leq \delta$, using $A_3 = 0$ and $A_4 <0$, taking $\epsilon$ as small and fixed, and after choosing a sufficiently small $\tau $, equality \eqref{w10} yields 
\begin{equation}\label{w}
\begin{array}{l}
\nu[M^n,  g, v^{m} dV_{g}] < \nu[\R^n,  dx^2, dV, m]. \\
\end{array}
\end{equation}

Proposition \ref{W y Yamabe tilda} implies
\begin{equation}\label{Lambda}
\begin{array}{l}
\Lambda[M^n,  g, v^{m} dV_{g}] < \Lambda[\R^n,  dx^2, dV, m]. \\
\end{array}
\end{equation}

Theorem \ref{Aubin Case} concludes the proof for $0 < m \leq \delta$. Finally, Theorem \ref{inductive constants} and an inductive argument imply that
\begin{equation}\label{inductivo}
\begin{array}{ll}
\Lambda[M^n,  g, v^{m+1} dV_{g}] & \leq  \dfrac{\Lambda[M^n,  g, v^{m} dV_{g}]}{\Lambda[\R^n,  dx^2, dV, m]}  \Lambda[\R^n,  dx^2,  dV, m+1] \vspace{0.2cm}\\
& < \Lambda[\R^n,  dx^2, dV, m+1],
\end{array}
\end{equation} which leads to our result for every $m \in  \bigcup \limits_{i \in \mathbb{N} \cup \{ 0 \}} [i,i + \delta)$. $\hspace{5cm} \square$

\begin{rem}\label{comentario n=6}
Note that the proof works if $A_5 \leq 0$, which is false for a general $m>0$.
\end{rem}

\section{Appendix}\label{appendix}

In this section, we show some calculus lemmas that we used in the proof of Theorem A. 

\begin{lema}\label{integral con simetria} 
Letting $1 \leq i,j \leq n$ with $i \neq j$, then
\begin{equation}
\begin{array}{rl}
\displaystyle \int_{B_{\epsilon}}  \frac{x_{i}^4 |x|^l }{(1 + \frac{c(m,n)}{\tau}|x|^2)^{k} } dx  & = 3 \displaystyle \int_{B_{\epsilon}}  \frac{x_{i}^2 x_j^2 |x|^l  }{(1 + \frac{c(m,n)}{\tau}|x|^2)^{k} } dx \vspace{0.2cm}\\
& = \displaystyle \frac{3}{n(n+2)} \int_{B_{\epsilon}}  \frac{|x|^{4 + l}  }{(1 + \frac{c(m,n)}{\tau}|x|^2)^{k} } dx. \\
\end{array}
\end{equation}
\end{lema}

\begin{proof} We will use the formula 
\[
\int_{S^{n-1}_{\rho}} q \,  dS_{\rho} = \frac{\rho^2}{d(d+n-2)} \int_{S^{n-1}_{\rho}} \Delta q \,  dS_{\rho},
\] where $S^{n-1}_{\rho}$ is the sphere of radius $\rho$ and $q$ is a homogeneous polynomial of degree $d$. Then
\[
\displaystyle \int_{S^{n-1}_{\rho}}  x_{i}^4   dS_{\rho}  = 3 \int_{S^{n-1}_{\rho}}  x_{i}^2 x_{j}^2   dS_{\rho} = \frac{3}{n(n+2)} \rho^4 \int_{S^{n-1}_{\rho}} dS_{\rho}.
\]

Using the last equality and polar coordinates, we obtain the result. \end{proof}

Next, we compare the integrals $I_1$, $I_2$, $I_3$, $I_4$, $I_5$, $I_6$, $I_7$, $I_8$ and $V$ considered in the above section. This kind of comparison appeared, for example, in \cite{Aubin} and \cite{EscobarAn}. 

\begin{lema}\label{comparacion integrales}
We obtain the following equalities 
\[
I_1 = \frac{4(m+n-1)(m+n-2)}{n(2m+n-4)} I_7, \quad I_2 = \frac{4(m+n-1)(m+n-2)}{(2m+n-4)(2m+n-6)}I_7,
\] 
\[
I_3 = \frac{n+2}{2m+n-4} I_7, \quad I_4 = \frac{(n+2)(n+4)}{(2m+n-4)(2m+n-6)} I_7,
\]
\[
 I_5 = \frac{2(m+n-1)}{2m+n-4} I_7, \quad I_6 =\frac{2(m+n-1)(n+2)}{(2m+n-4)(2m+n-6)}I_7, 
\]
\[
  I_8 = \frac{2(m+n-1)}{n}I_7, \quad \text{and} \quad \frac{I_7}{V} =  \frac{n c(m,n)^{\frac{n}{2}}}{2m+n-2}.
\]
\end{lema}

\begin{proof} Using polar coordinates, we obtain
\begin{equation}\label{I1 comparacion}
\begin{array}{ll}
I_1 = vol(S^{n-1}) \displaystyle \int_{0}^{\infty} \dfrac{r^{n-1}}{(1 + r^2)^{m+n-2}} dr,
\end{array}
\end{equation}
\begin{equation}\label{I2 comparacion}
\begin{array}{ll}
I_2  =   vol(S^{n-1}) \displaystyle \int_{0}^{\infty} \dfrac{r^{n+1}}{(1 + r^2)^{m+n-2}} dr,
\end{array}
\end{equation}
\begin{equation}\label{I3 comparacion}
\begin{array}{ll}
I_3 & =  vol(S^{n-1})  \displaystyle \int_{0}^{\infty} \dfrac{r^{n+3}}{(1 + r^2)^{m+n}} dr,
\end{array}
\end{equation}
\begin{equation}\label{I4 comparacion}
\begin{array}{ll}
I_4  =   vol(S^{n-1}) \displaystyle \int_{0}^{\infty} \dfrac{r^{n+5}}{(1 + r^2)^{m+n}} dr,
\end{array}
\end{equation}
\begin{equation}\label{I5 comparacion}
\begin{array}{ll}
I_5  =   vol(S^{n-1}) \displaystyle \int_{0}^{\infty} \dfrac{r^{n+1}}{(1 + r^2)^{m+n-1}} dr,
\end{array}
\end{equation}
\begin{equation}\label{I6 comparacion}
\begin{array}{ll}
I_6  =   vol(S^{n-1}) \displaystyle \int_{0}^{\infty} \dfrac{r^{n+3}}{(1 + r^2)^{m+n-1}} dr
\end{array}
\end{equation}
\begin{equation}\label{I7 comparacion}
\begin{array}{ll}
I_7  =   vol(S^{n-1}) \displaystyle \int_{0}^{\infty} \dfrac{r^{n+1}}{(1 + r^2)^{m+n}} dr,
\end{array}
\end{equation}
\begin{equation}\label{I8 comparacion}
\begin{array}{ll}
I_8  =   vol(S^{n-1}) \displaystyle \int_{0}^{\infty} \dfrac{r^{n-1}}{(1 + r^2)^{m+n-1}} dr,
\end{array}
\end{equation} and 
\begin{equation}\label{V comparacion}
\begin{array}{ll}
V  =   \dfrac{vol(S^{n-1})}{\tilde{c}(m, n)^{\frac{n}{2}}} \displaystyle  \int_{0}^{\infty} \dfrac{r^{n-1}}{(1 + r^2)^{m+n}} dr.
\end{array}
\end{equation}

Integrating by parts, we obtain for every $k> 1$, $l>1$, and $k>l$
\begin{equation}\label{comparacion}
\begin{array}{ll}
\displaystyle \int_{0}^{\infty} \dfrac{r^{l+1}}{(1 + r^2)^{k}} dr = \frac{l}{2(k-1)} \int_{0}^{\infty} \dfrac{r^{l-1}}{(1 + r^2)^{k-1}} dr,
\end{array}
\end{equation} which implies that $I_7 = \frac{n}{2(m+n-1)}I_8$, $I_3 = \frac{n+2}{2(m+n-1)}I_5$ and $I_4 = \frac{n+4}{2(m+n-1)}I_6$. To compare $I_5$ with $I_7$, we write
\begin{equation}\label{I9 Case comparacion}
\begin{array}{ll}
\displaystyle \int_{0}^{\infty} \dfrac{r^{n+1}}{(1 + r^2)^{m+n-1}} dr = \int_{0}^{\infty} \dfrac{r^{n+1}}{(1 + r^2)^{m+n}} dr + \int_{0}^{\infty} \dfrac{r^{n+3}}{(1 + r^2)^{m+n}} dr.
\end{array}
\end{equation}

Using equality \eqref{comparacion} in \eqref{I9 Case comparacion} yields 
\begin{equation}\label{I10 Case comparacion}
\begin{array}{ll}
\displaystyle \dfrac{2m+n-4}{2(m+n-1)}\int_{0}^{\infty} \dfrac{r^{n+1}}{(1 + r^2)^{m+n-1}} dr = \int_{0}^{\infty} \dfrac{r^{n+1}}{(1 + r^2)^{m+n}} dr.
\end{array}
\end{equation}

Hence, $I_5 = \frac{2(m+n-1)}{2m+n-4} I_7$ and $I_3 = \frac{n+2}{2m+n-4} I_7$. Similarly 
\begin{equation}\label{I2 Case comparacion orden 2}
\begin{array}{ll}
\displaystyle \int_{0}^{\infty} \dfrac{r^{n+3}}{(1 + r^2)^{m+n-1}} dr = \int_{0}^{\infty} \dfrac{r^{n+3}}{(1 + r^2)^{m+n}} dr + \int_{0}^{\infty} \dfrac{r^{n+5}}{(1 + r^2)^{m+n}} dr.
\end{array}
\end{equation}

Using equality \eqref{comparacion} in \eqref{I2 Case comparacion orden 2} yields 
\begin{equation}\label{I3 Case comparacion orden 2}
\begin{array}{ll}
\displaystyle \dfrac{2m+n-6}{2(m+n-1)}\int_{0}^{\infty} \dfrac{r^{n+3}}{(1 + r^2)^{m+n-1}} dr = \int_{0}^{\infty} \dfrac{r^{n+3}}{(1 + r^2)^{m+n}} dr.
\end{array}
\end{equation}

Then, \small $I_6 = \frac{2(m+n-1)}{2m+n-6} I_3 = \frac{2(m+n-1)(n+2)}{(2m+n-4)(2m+n-6)}I_7$ \normalsize and $I_4 = \frac{(n+2)(n+4)}{(2m+n-4)(2m+n-6)} I_7$.

To compare $I_7$ with $V$, we write 
\begin{equation}\label{I11 Case comparacion}
\begin{array}{ll}
\displaystyle \int_{0}^{\infty} \dfrac{r^{n-1}}{(1 + r^2)^{m+n-1}} dr = \int_{0}^{\infty} \dfrac{r^{n-1}}{(1 + r^2)^{m+n}} dr + \int_{0}^{\infty} \dfrac{r^{n+1}}{(1 + r^2)^{m+n}} dr.
\end{array}
\end{equation}

Using equality \eqref{comparacion} in equality above, we obtain
\[
\displaystyle \frac{2m+n-2}{n} \int_{0}^{\infty} \dfrac{r^{n+1}}{(1 + r^2)^{m+n}} dr = \int_{0}^{\infty} \dfrac{r^{n-1}}{(1 + r^2)^{m+n}} dr.
\]

Therefore,
\begin{equation}\label{I12 Case comparacion}
\begin{array}{ll}
  \dfrac{I_7}{V} = \dfrac{n c(m,n)^{\frac{n}{2}}}{2m+n-2}. 
\end{array}
\end{equation}

Now, we compare $I_1$ with $I_7$; for this purpose, observe that
\begin{equation}\label{I13 Case comparacion}
\begin{array}{ll}
\displaystyle \int_{0}^{\infty} \dfrac{r^{n-1}}{(1 + r^2)^{m+n-2}} dr = \int_{0}^{\infty} \dfrac{r^{n-1}}{(1 + r^2)^{m+n-1}} dr + \int_{0}^{\infty} \dfrac{r^{n+1}}{(1 + r^2)^{m+n-1}} dr.
\end{array}
\end{equation}

Hence, $I_1 = I_8 + I_5$. Therefore, $I_1 = \frac{4(m+n-1)(m+n-2)}{n(2m+n-4)} I_7$. It remains to compare $I_2 $ with $I_7 $. We have
\begin{equation}\label{I13 Case comparacion orden 2}
\begin{array}{ll}
\displaystyle \int_{0}^{\infty} \dfrac{r^{n+1}}{(1 + r^2)^{m+n-2}} dr = \int_{0}^{\infty} \dfrac{r^{n+1}}{(1 + r^2)^{m+n-1}} dr + \int_{0}^{\infty} \dfrac{r^{n+3}}{(1 + r^2)^{m+n-1}} dr.
\end{array}
\end{equation}

It follows from equalities \eqref{I13 Case comparacion orden 2} and \eqref{comparacion} that $I_2 = \frac{n}{2(m+n-2)}I_5 + I_6$. As a consequence, $I_2 = \frac{4(m+n-1)(m+n-2)}{(2m+n-4)(2m+n-6)}I_7$. \end{proof}

\section*{Acknowledgements}

I am grateful to my advisor, Fernando Cod\'a Marques, for guidance and support. I am also grateful to the Universidad del Valle, Cali, Colombia, for support during my Ph.D. studies and through project CI 71205; to Capes for a  scholarship provided during Ph.D. studies; and to the Mathematics Department of Princeton University for its hospitality during my visits in October 2014 and September-October 2015 at the beginning of this process. Additionally, I want to thank IMPA for its hospitality in January-February 2020, when I wrote the second version of this document.

\end{document}